\documentclass[a4paper]{article}

\DeclareFontShape{OT1}{cmr}{m}{scit}{<->ssub*cmr/m/sc}{}
\usepackage{slantsc}

\usepackage[margin=1in]{geometry}
\usepackage{amsmath}
\usepackage{amsthm}
\usepackage{aliascnt}
\usepackage{amssymb}
\usepackage{graphicx}
\usepackage{booktabs}
\usepackage{multirow}
\usepackage{xcolor}
\usepackage[utf8]{inputenc}
\usepackage{adjustbox}
\usepackage[mode=tex]{standalone}
\usepackage{microtype}
\usepackage{enumitem}
\usepackage{verbatim}
\usepackage{array}
\usepackage{tabularx}
\usepackage{arydshln}
\usepackage{tikz}
\usetikzlibrary{shapes,calc,math,backgrounds,matrix}

\usepackage[%
backend=biber,
bibstyle=numeric,
citestyle=numeric-comp,
maxcitenames=3,
maxbibnames=20,
sorting=ydnt
]{biblatex}
\usepackage{hyperref}
\hypersetup{
	colorlinks,
	allcolors=blue,
	unicode,
	bookmarksopen,
	bookmarksdepth=2
}

\usepackage[capitalise,noabbrev]{cleveref}
\crefname{cond}{Condition}{Conditions}
\creflabelformat{cond}{(#2C#1#3)}
\crefname{eitem}{}{}
\creflabelformat{eitem}{(#2#1#3)}
\crefname{step}{\textbf{Step}}{\textbf{Steps}}
\creflabelformat{step}{{\bf #2#1#3}}

\theoremstyle{plain}
\newtheorem{theorem}{Theorem}[section]
\newaliascnt{corollary}{theorem}
\newtheorem{corollary}[corollary]{Corollary}
\aliascntresetthe{corollary}
\newaliascnt{lemma}{theorem}
\newtheorem{lemma}[lemma]{Lemma}
\aliascntresetthe{lemma}

\newaliascnt{axiom}{theorem}

\aliascntresetthe{axiom}
\newaliascnt{conjecture}{theorem}

\aliascntresetthe{conjecture}
\newaliascnt{fact}{theorem}

\aliascntresetthe{fact}
\newaliascnt{hypothesis}{theorem}

\aliascntresetthe{hypothesis}
\newaliascnt{assumption}{theorem}

\aliascntresetthe{assumption}
\newaliascnt{proposition}{theorem}
\newtheorem{proposition}[proposition]{Proposition}
\aliascntresetthe{proposition}
\newaliascnt{criterion}{theorem}

\aliascntresetthe{criterion}
\theoremstyle{definition}
\newaliascnt{definition}{theorem}
\newtheorem{definition}[definition]{Definition}
\aliascntresetthe{definition}
\newaliascnt{example}{theorem}

\aliascntresetthe{example}
\newaliascnt{remark}{theorem}
\newtheorem{remark}[remark]{Remark}
\aliascntresetthe{remark}
\newaliascnt{principle}{theorem}

\aliascntresetthe{principle}

\crefname{theorem}{Theorem}{Theorems}
\Crefname{theorem}{Theorem}{Theorems}
\crefname{corollary}{Corollary}{Corollaries}
\Crefname{corollary}{Corollary}{Corollaries}
\crefname{lemma}{Lemma}{Lemmas}
\Crefname{lemma}{Lemma}{Lemmas}
\crefname{claim}{Claim}{Claims}
\Crefname{claim}{Claim}{Claims}
\crefname{axiom}{Axiom}{Axioms}
\Crefname{axiom}{Axiom}{Axioms}
\crefname{conjecture}{Conjecture}{Conjectures}
\Crefname{conjecture}{Conjecture}{Conjectures}
\crefname{fact}{Fact}{Facts}
\Crefname{fact}{Fact}{Facts}
\crefname{hypothesis}{Hypothesis}{Hypotheses}
\Crefname{hypothesis}{Hypothesis}{Hypotheses}
\crefname{assumption}{Assumption}{Assumptions}
\Crefname{assumption}{Assumption}{Assumptions}
\crefname{proposition}{Proposition}{Propositions}
\Crefname{proposition}{Proposition}{Propositions}
\crefname{criterion}{Criterion}{Criteria}
\Crefname{criterion}{Criterion}{Criteria}
\crefname{definition}{Definition}{Definitions}
\Crefname{definition}{Definition}{Definitions}
\crefname{example}{Example}{Examples}
\Crefname{example}{Example}{Examples}
\crefname{remark}{Remark}{Remarks}
\Crefname{remark}{Remark}{Remarks}
\crefname{principle}{Principle}{Principles}
\Crefname{principle}{Principle}{Principles}

\newcommand{\email}[1]{\href{mailto:#1}{\texttt{#1}}}

\newcommand{\TSsym}{\mathsf{TS}}

\newcommand{\TS}[2]{\TSsym_{#1}(#2)}

\newcommand{\defclassprob}[3]{%
	\par\vspace{1mm}%
	\noindent\fbox{%
		\begin{minipage}{\dimexpr\textwidth-2\fboxsep-2\fboxrule\relax}
			#1 \\
			\textbf{Fixed data:} #2 \\
			\textbf{Objective:} #3
		\end{minipage}%
	}%
	\par
}

\numberwithin{equation}{section}

\newcommand{\GTSRealizability}{\textsc{$\mathcal{G}$-$\TSsym_k$-Realizability}}
\newcolumntype{L}[1]{>{\raggedright\arraybackslash}p{#1}}
\newcolumntype{Y}{>{\raggedright\arraybackslash}X}
\newcommand{\TableGroupRule}{\specialrule{1.1pt}{0.45ex}{0.45ex}}

\title{\textbf{Token-sliding realizability for complements, Cartesian-products, and grid graph families}}
\author{Duc A.\ Hoang\\
VNU University of Science, Vietnam National University\\
Hanoi, Vietnam\\
\email{hoanganhduc@hus.edu.vn}}
\date{}

\begin{document}
\maketitle

\begin{abstract}
For an integer $k\ge 0$ and a graph $G$, the \emph{token-sliding reconfiguration graph $\mathsf{TS}_k(G)$} has the independent $k$-sets of $G$ as vertices.
Two vertices are adjacent if one token can slide along an edge of $G$ and the resulting $k$-set is still independent.
We study the following realizability problem: for fixed $k\ge 2$, which graphs are isomorphic to $\mathsf{TS}_k(G)$ for some graph $G$?
This inverse viewpoint asks which abstract state spaces can occur exactly under a local token rule.
We give positive realizability results for the complement targets $\overline{K_n}$, $\overline{K_{m,n}}$, and $\overline{K_n-e}$, and we determine sharp cutoffs for complements of paths and cycles.
We also prove a product formula for token-sliding graphs of disjoint unions and apply it to Cartesian products of complete graphs, paths, and cycles.
For every grid $\Gamma_{m,n}=P_m\square P_n$ with $2\le m\le n$, we realize $\Gamma_{m,n}$ at token value $m+n-2$ and at every token value $k\ge 4$.
At small token values, we prove that $C_4\square C_n$ is not a $\mathsf{TS}_2$-graph for $n\ge 4$, classify ladders $\Gamma_{2,n}$, and settle the first non-ladder grid: for $k\ge 2$, $\Gamma_{3,3}$ is realizable if and only if $k\ge 4$.

\end{abstract}

\section{Introduction}

Recently, \emph{reconfiguration problems} have attracted attention from both theoretical and practical viewpoints.
Given a computational problem $\mathcal{P}$ and a \emph{reconfiguration rule}, the corresponding \emph{reconfiguration graph} has the feasible solutions of $\mathcal{P}$ as vertices.
Two vertices are \emph{adjacent} when one feasible solution is obtained from the other by one application of the rule.
For an overview of this research area, we refer readers to the surveys~\cite{Nishimura2018,vandenHeuvel2013,MynhardtN2019,BousquetMNS2024}.

The \textsc{Token-Sliding} model for the \textsc{Independent Set} problem was introduced in the Nondeterministic Constraint Logic framework of Hearn and Demaine~\cite{HearnDemaine2005}.
A fixed-size independent set of a graph $G$ is viewed as a set of tokens placed on the vertices of $G$.
The \textsc{Token-Sliding} rule allows one token to slide along an edge of $G$ to an unoccupied vertex, provided that the resulting set of tokens is again independent.
We denote the corresponding reconfiguration graph by $\TSsym_k(G)$.

Rather than starting with a graph $G$ and asking how its independent $k$-sets reconfigure, we start with a target graph $H$ and ask whether $H$ can occur exactly as some $\TSsym_k(G)$.
A graph $H$ is a \emph{$\TSsym_k$-graph} if $H\cong \TSsym_k(G)$ for some graph~$G$.
For a graph family $\mathcal{G}$, the corresponding \emph{family-realizability problem} asks which members of $\mathcal{G}$ are $\TSsym_k$-graphs.
We call this the \emph{\GTSRealizability{} problem}.

\defclassprob{\GTSRealizability}{An integer $k\ge 2$ and a graph family $\mathcal{G}$.}{Classify the graphs $H\in\mathcal{G}$ for which there exists a graph $X$ such that $\TSsym_k(X)\cong H$.}

In this paper, we study \GTSRealizability{} through complement targets, a disjoint-union product formula, and grids.
Complement targets expose local obstructions and sharp cutoff behavior, while the product formula gives Cartesian-product targets.
Grid graphs are the main product test family: they combine Cartesian-product structure, many induced $4$-cycles, and strong local constraints, exposing small-token obstructions.
We use the grid notation $\Gamma_{m,n}=P_m\square P_n$ defined in \cref{sec:prelim}; $\Gamma_{2,n}$ is a \emph{ladder}.
Whenever we discuss grids, we adopt the standing assumption $2\le m\le n$.
The path case $\Gamma_{1,n}\cong P_n$ is excluded by this convention; it is already settled in~\cite[Cor.~5]{AvisHoang2023}.

Avis and Hoang~\cite{AvisHoang2023,AvisHoang2024} established initial target-class realizability results, construction and closure tools, decomposition methods, and acyclic tree/forest constructions.
See \Cref{app:known-realizability} for a summary of these results.
Some elementary complement-positive cases follow from those results together with complement identities, but the sharp complement cutoffs and the Cartesian-product and grid families treated here are not settled there.

The realizability viewpoint is adjacent to several other state-graph models.
Token graphs allow all $k$-subsets of a vertex set as states, while matching graphs restrict the host to line graphs and the states to maximum matchings.
Clique reconfiguration gives a parallel target-realizability problem for clique states rather than independent-set states.
These neighboring models place the problem in context.
Here we ask which prescribed graphs arise exactly as token-sliding graphs of independent sets, with no restriction beyond finite simple host graphs.

The results form a construction--obstruction program rather than a list of unrelated family classifications.
We now summarize the main results; see \cref{tab:realizability-summary} for the results proved here and the partial product frontiers settled in this paper.

\begin{itemize}
\item \emph{Complements.}
We settle several standard complement families: $\overline{K_n}$, $\overline{K_{m,n}}$, $\overline{K_n-e}$, $\overline{P_n}$, and $\overline{C_n}$.
The first three families are positive for every fixed $k\ge 2$ (\cref{thm:edgeless,thm:complement-complete-bipartite,thm:complement-complete-minus-edge}).
For $\overline{K_n}$ we give a direct construction; the other two families are short consequences of prior complete-graph realizability and disjoint-union closure.
For complements of paths and cycles, the cutoffs are sharp: $\overline{P_n}$ is realizable exactly for $n\le 5$, and $\overline{C_n}$ exactly for $n\le 6$ (\cref{thm:complement-paths,thm:complement-cycles}).

\item \emph{Products.}
We prove a disjoint-union product formula for token-sliding graphs, yielding Cartesian-product families (\cref{thm:product-lemma,cor:product-saturation,cor:product-families}).
At two tokens, this positive behavior has limits: $C_4\square C_n$ is not a $\TSsym_2$-graph for $n\ge 4$ (\cref{thm:c4-cycle-products-not-ts2}).

\item \emph{Grids and ladders.}
For every $2\le m\le n$, the grid $\Gamma_{m,n}$ is a $\TSsym_{m+n-2}$-graph and, by a separate four-token construction with padding, a $\TSsym_k$-graph for every $k\ge 4$ (\cref{thm:grid-realizability,cor:grid-allk,thm:grid-four-token-product,thm:grid-allk-four}).
We also completely classify ladders by fixed token number: for $n\ge 2$ and $k\ge 2$, the ladder $\Gamma_{2,n}$ is a $\TSsym_k$-graph if and only if $k\ge 3$ or $(n,k)=(2,2)$ (\cref{thm:ladder-bounds}).
The first non-ladder width-three grid has the opposite small-token behavior: for $k\ge 2$, $\Gamma_{3,3}$ is a $\TSsym_k$-graph if and only if $k\ge 4$ (\cref{thm:gamma33-fixed-k}).
\end{itemize}

\begin{table}[ht]
\centering
\renewcommand{\arraystretch}{1.18}
\begin{tabularx}{\textwidth}{@{}L{0.16\textwidth}L{0.24\textwidth}YL{0.20\textwidth}@{}}
\toprule
Group & Target graph $H$ & $\TSsym_k$-realizable? & Reference \\
\TableGroupRule
\multirow[c]{5}{0.16\textwidth}{Complements}
& $\overline{K_n}$, $n\ge 1$ & yes for all $k\ge 2$ & Thm.~\labelcref{thm:edgeless} \\
\cmidrule(lr){2-4}
& $\overline{K_{m,n}}$, $m,n\ge 1$ & yes for all $k\ge 2$ & Thm.~\labelcref{thm:complement-complete-bipartite} \\
\cmidrule(lr){2-4}
& $\overline{K_n-e}$, $n\ge 2$ & yes for all $k\ge 2$ & Thm.~\labelcref{thm:complement-complete-minus-edge} \\
\cmidrule(lr){2-4}
& $\overline{P_n}$, $n\ge 1$ & yes iff $n\le 5$ & Thm.~\labelcref{thm:complement-paths} \\
\cmidrule(lr){2-4}
& $\overline{C_n}$, $n\ge 3$ & yes iff $n\le 6$ & Thm.~\labelcref{thm:complement-cycles} \\
\TableGroupRule
\multirow[c]{10}{0.16\textwidth}{Products}
& $K_n\square K_m$, $m,n\ge 1$ & yes for all $k\ge 2$ & Cor.~\labelcref{cor:product-families} \\
\cmidrule(lr){2-4}
& $P_n\square K_m$, $m,n\ge 1$ & yes for all $k\ge 3$ & Cor.~\labelcref{cor:product-families} \\
\cmidrule(lr){2-4}
& $C_n\square K_m$, $n\ge 3$, $m\ge 1$ & yes for all $k\ge 3$ & Cor.~\labelcref{cor:product-families} \\
\cmidrule(lr){2-4}
& $C_m\square C_n$, $m,n\ge 3$ & yes for all $k\ge 4$ & Cor.~\labelcref{cor:product-families} \\
\cmidrule(lr){2-4}
& $C_4\square C_n$, $n\ge 4$ & no for $k=2$ & Thm.~\labelcref{thm:c4-cycle-products-not-ts2} \\
\cmidrule(lr){2-4}
& \multirow[c]{2}{0.24\textwidth}{$C_m\square P_n$, $m\ge 3$, $n\ge 1$} & yes for all $k\ge 2$ when $n=1$; yes for $(m,n,k)=(3,2,2)$; yes for $k=3$ with $m=3$ or $n\le 2$; and yes for all $k\ge 4$ & Cor.~\labelcref{cor:cycle-path-positive} \\
\cmidrule(lr){3-4}
& & no for $k=2$ when $m\ge 4$ and $n=2$, or when $m=3$ and $3\le n\le 5$ & Prop.~\labelcref{prop:cycle-path-k2-structural} \\
\cmidrule(lr){2-4}
& $\Gamma_{m,n} = P_m \square P_n$, $2\le m\le n$ & yes for $k=m+n-2$ and all $k\ge 4$ & Thms.~\labelcref{thm:grid-realizability,thm:grid-allk-four} \\
\cmidrule(lr){2-4}
& $\Gamma_{2,n}$, $n\ge 2$ & yes iff $k\ge 3$ or $(n,k)=(2,2)$ & Thm.~\labelcref{thm:ladder-bounds} \\
\cmidrule(lr){2-4}
& $\Gamma_{3,3}$ & yes iff $k\ge 4$ & Thm.~\labelcref{thm:gamma33-fixed-k} \\
\bottomrule
\end{tabularx}
\caption{Summary of $\TSsym_k$-realizability results proved in this paper, with $k\ge 2$. Some rows are complete classifications, while product rows record the positive ranges and small-token exclusions settled here. Broader product rows and sharper small-token rows may overlap.}
\label{tab:realizability-summary}
\end{table}

For $C_m\square P_n$, \cref{tab:realizability-summary} records only the cases settled here.

The proofs combine explicit host constructions, complement-root obstructions, a disjoint-union product formula, padding, and small-token obstructions.
Together, these tools show that token-sliding target spaces have strong product closure at larger token values but sharp local obstructions at small token values.

The realizability viewpoint is related to, but different from, the theory of \emph{token graphs} $F_k(G)$, where all $k$-subsets of $V(G)$ are allowed as configurations~\cite{FabilaMonroyFHHUW2012}.
Directed variants of token graphs have also been studied recently~\cite{FernandesLPSZ2026}.
Related reconfiguration graphs have been studied for \emph{proper colorings} and \emph{dominating sets}~\cite{MynhardtN2019,MessingerPorter2025,MessingerPipes2026}.

\emph{Matching graphs}~\cite[pp.~73--74]{ErohSchultz1998} give a close direct comparison.
The matching graph of a graph $R$ has the maximum matchings of $R$ as vertices, with adjacency when two matchings differ by one edge.
Since a maximum matching of $R$ is a maximum independent set of the line graph $L(R)$, and a one-edge exchange between maximum matchings must exchange incident edges, the matching graph of $R$ is canonically isomorphic to $\TS{\alpha(L(R))}{L(R)}$.
Thus, for fixed $k$, matching-graph realizability with maximum matchings of size~$k$ is a restricted case of \GTSRealizability{}.
The hosts are line graphs, and the token sets are maximum independent sets.
The foundational matching-graph literature studies constructions and characterizations for standard target families, including paths, stars, cycles, complete graphs, hypercubes, and trees~\cite{JonesRoehmSchultz1998,ErohSchultz1998,Liu2005MaxMatchingTypes}.
Further target-family results and restrictions for maximum matching graphs appear in~\cite{WangYuanLiuLin1998,Liu2004Subgraphs}.
Second-kind maximum matching graphs have also been studied~\cite{LiuLiu2014SecondKind}, but their adjacency rule is different and is not the independent-set token-sliding graph of a line-graph host.
Some maximum-matching-graph results are listed in \cref{app:known-realizability}.
They motivate comparison with our path, cycle, acyclic, and product constructions, but matching graphs remain formally more restrictive.

Clique reconfiguration gives another parallel target-realizability problem.
In that setting, states are $k$-cliques; the token-sliding and token-jumping rules replace one clique vertex by an adjacent non-member or by an arbitrary non-member, respectively.
Lam, Phan, and Hoang~\cite{LamPhanHoang2026} studied structural properties of these clique reconfiguration graphs, and Hoang~\cite{Hoang2026} studied the corresponding target-realizability problem.
The analogy is structural: the feasible states are cliques rather than independent sets.

The rest of this paper is organized as follows.
\Cref{sec:prelim} gives definitions and preliminary lemmas.
\Cref{sec:basic-complements} treats complements of standard graph families, after recalling the complete, path, and cycle realizations used in the proofs.
\Cref{sec:product-construction} proves the disjoint-union product formula and applies it to products of complete graphs, paths, and cycles.
\Cref{sec:grid-realizability} applies these tools to grids and proves the fixed-token classifications of ladders and $\Gamma_{3,3}$.
\Cref{sec:conclusion} gives concluding remarks.
\section{Preliminaries}
\label{sec:prelim}

\paragraph{Graph invariants and notation.}
Throughout the paper, all graphs are finite, simple, and undirected.
All realizability statements have token-set size $k\ge 2$ unless explicitly stated otherwise.
We use $V(G)$ and $E(G)$ for the vertex and edge sets of a graph~$G$.
The \emph{neighborhood} of a vertex $v$ in $G$ is $N_G(v):=\{w\in V(G):vw\in E(G)\}$, and its \emph{closed neighborhood} is $N_G[v]:=N_G(v)\cup\{v\}$; for $R\subseteq V(G)$, let $\displaystyle N_G[R]:=\bigcup_{r\in R}N_G[r]$.
Let $\deg_G(v):=|N_G(v)|$ denote the \emph{degree} of~$v$, and let $\Delta(G)$, $\alpha(G)$, and $\omega(G)$ denote the \emph{maximum degree}, \emph{independence number}, and \emph{clique number}, respectively.
We use the standard identity $\alpha(\overline{G})=\omega(G)$.
A vertex of degree~$1$ is a \emph{pendant vertex} or \emph{leaf}, and its unique incident edge is a \emph{pendant edge}.
A vertex is \emph{universal} if it is adjacent to every other vertex of the graph.
When the graph~$G$ is clear from context, we drop the subscript.
For a graph $G$, let $t(G)$ denote its \emph{triangle count}, the number of triangles of~$G$; for an edge $xy\in E(G)$, let $t_G(xy):=|N_G(x)\cap N_G(y)|$ denote the number of triangles containing~$xy$.
If $G$ is connected, let $r(G):=|E(G)|-|V(G)|+1$ denote its \emph{cycle rank}.
Equivalently, if $T$ is a \emph{spanning tree} of~$G$, then $r(G)=|E(G)\setminus E(T)|$; thus $r(G)$ counts how many edges must be deleted from~$G$ to obtain a tree.
An edge of a graph is a \emph{triangle-edge} if it lies in at least one triangle, and the \emph{triangle-edge set} of $G$ is the set of all triangle-edges of~$G$.

\paragraph{Graph classes and operations.}
For $W\subseteq V(G)$, let $G[W]$ denote the subgraph of~$G$ induced by~$W$.
For a vertex $v\in V(G)$, $G-v$ denotes the graph obtained from~$G$ by deleting~$v$ and all incident edges; for $W\subseteq V(G)$, $G-W$ is defined analogously.
For an edge set $X\subseteq E(G)$, let $G-X$ denote the graph obtained from~$G$ by deleting the edges of~$X$.
For graphs $G_1,G_2$ with disjoint vertex sets, $G_1\sqcup G_2$ denotes their \emph{disjoint union}.
For a finite set $V$, let $\displaystyle \binom{V}{2}$ denote the set of all two-element subsets of~$V$.
For a graph $G=(V,E)$, its \emph{complement} is $\displaystyle \overline{G}=(V,\binom{V}{2}\setminus E)$.
For a graph~$H$, its \emph{line graph} $L(H)$ has vertex set $E(H)$, with two vertices adjacent if the corresponding edges of~$H$ share an endpoint.
The \emph{Cartesian product} $G\square H$ has vertex set $V(G)\times V(H)$; two vertices $(g,h)$ and $(g',h')$ are adjacent if either $g=g'$ and $hh'\in E(H)$, or $h=h'$ and $gg'\in E(G)$.
The \emph{path}, \emph{cycle}, \emph{complete graph}, and \emph{complete bipartite graph} on the indicated number of vertices are denoted $P_n$, $C_n$, $K_n$, and $K_{m,n}$, respectively.
For integers $m,n\ge 1$, let $\Gamma_{m,n}:=P_m\square P_n$ denote the \emph{$(m\times n)$-grid graph}, and call $\Gamma_{2,n}$ a \emph{ladder}.
Whenever we discuss grid targets, we use the standing convention $2\le m\le n$ unless stated otherwise.
A \emph{triangle} is a $K_3$.
If $e$ is an edge of $K_n$, then $K_n-e$ denotes the graph obtained from $K_n$ by deleting~$e$.
For $r\ge 0$, the notation $rK_1$ means the disjoint union of $r$ isolated vertices; when $r=0$ this graph is empty.
A \emph{claw} is the graph $K_{1,3}$; a graph is \emph{claw-free} if it contains no induced claw.
We use two elementary facts.
First, line graphs are claw-free: if one edge is incident with three other edges, then two of those three share an endpoint and hence are adjacent in the line graph.
Second, under the standing assumption $2\le m\le n$, the grid $\Gamma_{m,n}$ contains an induced claw whenever $(m,n)\ne(2,2)$: take any grid vertex of degree at least~$3$ and any three of its grid neighbors, which are pairwise non-adjacent because grids are bipartite.

\paragraph{Token-sliding and realizability notation.}
For a graph $G=(V,E)$ and an integer $s\ge 0$, the \emph{token-sliding reconfiguration graph} $\TSsym_s(G)$ has vertex set consisting of all independent sets $I\subseteq V(G)$ with $|I|=s$.
Two vertices $I,J$ of $\TSsym_s(G)$ are adjacent if there exist $u\in I\setminus J$ and $v\in J\setminus I$ such that $I\setminus J=\{u\}$, $J\setminus I=\{v\}$, and $uv\in E(G)$.
If $G$ has no independent set of size~$s$, then $\TSsym_s(G)$ is the empty graph on zero vertices.
By convention, $\TSsym_0(G)$ is the one-vertex graph whose unique vertex is the empty independent set.
Thus decomposition formulas may include a zero-token factor.
\Cref{fig:ts2-c5} illustrates $\TSsym_2(C_5)$.
A graph $H$ is a \emph{$\TSsym_k$-graph} if $H\cong \TSsym_k(G)$ for some graph $G$.
For a token move from an independent set $I$ to an independent set $J$ in $\TSsym_k(G)$, we call $xy$ the \emph{slide edge} if the removed token $x\in I\setminus J$ slides along the base-graph edge $xy$ to the added token $y\in J\setminus I$.
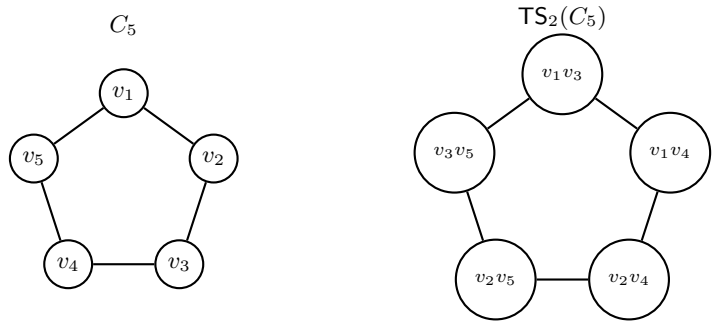
\begin{figure}[ht]
\centering
\begin{adjustbox}{max width=\textwidth}
\begin{tikzpicture}[
	vertex/.style={circle,draw,thick,fill=white,inner sep=1pt,minimum size=18pt,font=\small},
	state/.style={circle,draw,thick,fill=white,inner sep=1pt,minimum size=30pt,font=\scriptsize},
	edge/.style={thick},
	title/.style={font=\small\bfseries}
]
\node[title] at (0,2.15) {$C_5$};
\node[vertex] (v1) at (90:1.25) {$v_1$};
\node[vertex] (v2) at (18:1.25) {$v_2$};
\node[vertex] (v3) at (-54:1.25) {$v_3$};
\node[vertex] (v4) at (-126:1.25) {$v_4$};
\node[vertex] (v5) at (162:1.25) {$v_5$};
\draw[edge] (v1) -- (v2) -- (v3) -- (v4) -- (v5) -- (v1);

\node[title] at (5.8,2.25) {$\TSsym_2(C_5)$};
\node[state] (s13) at ($(5.8,0)+(90:1.50)$) {$v_1v_3$};
\node[state] (s14) at ($(5.8,0)+(18:1.50)$) {$v_1v_4$};
\node[state] (s24) at ($(5.8,0)+(-54:1.50)$) {$v_2v_4$};
\node[state] (s25) at ($(5.8,0)+(-126:1.50)$) {$v_2v_5$};
\node[state] (s35) at ($(5.8,0)+(162:1.50)$) {$v_3v_5$};
\draw[edge] (s13) -- (s14) -- (s24) -- (s25) -- (s35) -- (s13);
\end{tikzpicture}
\end{adjustbox}
\caption{The token-sliding graph $\TSsym_2(C_5)$. The label $v_iv_j$ denotes $\{v_i,v_j\}$; edges are single-token slides in $C_5$.}
\label{fig:ts2-c5}
\end{figure}

\paragraph{Tools from Avis--Hoang.}
We use the following results of Avis and Hoang~\cite{AvisHoang2023}.

\begin{lemma}[{\cite[Lem.~1]{AvisHoang2023}}]
\label{lem:line-avis}
Let $H$ be a graph and let $G:=\overline{H}$.
Then, under the natural identification of independent $2$-sets of $G$ with edges of $H$, the graph $\TSsym_2(G)$ is a spanning subgraph of the line graph $L(H)$.
If $H$ is triangle-free, then $\TSsym_2(G)\cong L(H)$.
\end{lemma}

The padding result lets us raise the token number by adding isolated vertices.

\begin{proposition}[{\cite[Prop.~3]{AvisHoang2023}}]
\label{prop:padding}
Let $H$ be a graph with $\alpha(H)\ge 1$, and let $k\ge\alpha(H)$.
Let $H'$ be obtained from $H$ by adding $k-\alpha(H)$ isolated vertices.
Then $\TSsym_k(H')\cong \TSsym_{\alpha(H)}(H)$.
\end{proposition}

\paragraph{Reductions and size bounds.}

When a two-token realization has the form $\TSsym_2(\overline{F})$, we call $F$ a \emph{complement root} of the realization.
If universal vertices have first been deleted from the original host without changing the token-sliding graph, we call the resulting complement root a \emph{reduced complement root}.

Two elementary reductions will keep later two-token roots small.

\begin{lemma}
\label{lem:universal}
Let $G$ be a graph, let $u$ be a universal vertex of~$G$, and let $k\ge 2$.
Then $\TSsym_k(G)\cong \TSsym_k(G-u)$.
\end{lemma}

\begin{proof}
We show this by comparing the token states before and after deleting the universal vertex.
A universal vertex cannot belong to any independent set of size~$\ge 2$.
Hence, the vertex sets of $\TSsym_k(G)$ and $\TSsym_k(G-u)$ coincide, and so do the allowed token-slides.
\end{proof}

Non-empty connected two-token graphs can be realized with a connected complement root of controlled size.

\begin{lemma}
\label{lem:sizebound}
Let $G$ be a graph and set $N:=|V(\TSsym_2(G))|$.
If $\TSsym_2(G)$ is non-empty and connected, then there exists a graph~$G'$ with $\TSsym_2(G')\cong \TSsym_2(G)$, $|V(G')|\le N+1$, and $\overline{G'}$ connected with no isolated vertices and exactly $N$ edges.
\end{lemma}

\begin{proof}
To prove the size bound, we first remove irrelevant universal vertices and then pass to the complement root.
Delete all universal vertices of~$G$ iteratively; by \cref{lem:universal} this does not change $\TSsym_2(G)$.
Let $G'$ be the resulting graph and put $H:=\overline{G'}$.
By construction, $H$ has no isolated vertices, and $|E(H)|=N$ because vertices of $\TSsym_2(G')$ are exactly the edges of~$H$.

By \cref{lem:line-avis}, $\TSsym_2(G')$ is a spanning subgraph of $L(H)$ with vertex set $E(H)$.
If $H$ were disconnected, then $L(H)$ would be disconnected because $H$ has no isolated vertices.
Deleting edges from $L(H)$ cannot reconnect components, so $\TSsym_2(G')$ would also be disconnected, contradicting the assumption.
Therefore, $H$ is connected.
A connected graph with $N$ edges has at most $N+1$ vertices, so $|V(G')|=|V(H)|\le N+1$.
\end{proof}

\paragraph{Complement-root local tools.}

Endpoint degrees in a complement root determine degrees in the two-token graph.

\begin{proposition}
\label{prop:per-edge-degree}
Let $F$ be a simple graph and let $ab\in E(F)$ be identified with a vertex of $\TSsym_2(\overline{F})$.
Then
\[
 \deg_{\TSsym_2(\overline{F})}(ab)\;=\;\deg_F(a)+\deg_F(b)-2-2t_F(ab).
\]
\end{proposition}

\begin{proof}
We show the formula by counting the possible slides from the two endpoints of~$ab$ separately.
By the definition of token sliding, a $\TSsym_2(\overline{F})$-neighbor of~$ab$ has the form $\{a,w\}$ with $w\in N_F(a)\setminus N_F[b]$ or the form $\{b,w\}$ with $w\in N_F(b)\setminus N_F[a]$.
These two families are disjoint: equality $\{a,w\}=\{b,w'\}$ with $a\ne b$ would force $w=b$ and $w'=a$, both excluded by the displayed set differences.
For the first family, $w$ ranges over $N_F(a)\setminus N_F[b]$.
Since $N_F[b]=\{b\}\cup N_F(b)$, we subtract one for $b\in N_F(a)$ and $t_F(ab)$ for the common neighbors $w\in N_F(a)\cap N_F(b)$, giving $|N_F(a)\setminus N_F[b]|=\deg_F(a)-1-t_F(ab)$.
The second family contributes $\deg_F(b)-1-t_F(ab)$ similarly.
Summing gives the claimed identity.
\end{proof}

Triangle-freeness has a simple complement-root form.

\begin{proposition}
\label{prop:claw-iff-triangle-free}
For every graph $F$, $\TSsym_2(\overline{F})$ is triangle-free if and only if $F$ is claw-free.
\end{proposition}

\begin{proof}
We prove the two directions separately.
\begin{itemize}
\item[$(\Rightarrow)$]
Suppose $\TSsym_2(\overline{F})$ is triangle-free.
First, note that no vertex of $F$ has three pairwise non-adjacent neighbors.
Indeed, if $u_1,u_2,u_3$ were pairwise non-adjacent neighbors of some vertex $v$, then the three edges $vu_1,vu_2,vu_3$ would be pairwise adjacent in $\TSsym_2(\overline{F})$, forming a triangle.
Thus, $F$ has no induced claw.

\item[$(\Leftarrow)$]
Suppose, toward a contradiction, that $F$ is claw-free and $\TSsym_2(\overline{F})$ contains a triangle on vertices $\sigma_1,\sigma_2,\sigma_3\in E(F)$.
Each pair $\sigma_i,\sigma_j$ is adjacent in $\TSsym_2(\overline{F})$, so the pair shares an endpoint and the two non-shared endpoints are non-adjacent in~$F$.
Three pairwise incident edges in~$F$ form either a triangle of~$F$ or a star $K_{1,3}$ at a common endpoint.
The triangle case is impossible, because two edges of a triangle in~$F$ have adjacent non-shared endpoints and hence are not adjacent in $\TSsym_2(\overline{F})$.
In the star case, the three non-shared endpoints are pairwise non-adjacent by the token-sliding adjacency condition for the three edge pairs.
Therefore, the three edges form an induced claw in~$F$.
This contradicts claw-freeness, so $\TSsym_2(\overline{F})$ is triangle-free.
\end{itemize}
\end{proof}

\section{Complements of Some Basic Graphs}
\label{sec:basic-complements}

This section treats complements of several standard graph families.
We start with the edgeless graphs $\overline{K_n}$ and recall the complete, path, and cycle targets used later.
We then use disjoint-union closure for complements of complete bipartite graphs and complements of $K_n-e$.
Finally, we classify complements of paths and cycles.

\begin{theorem}
\label{thm:edgeless}
For every fixed $k\ge 2$ and every $n\ge 1$, the edgeless graph $\overline{K_n}$ is a $\TSsym_k$-graph.
\end{theorem}

\begin{proof}
We build a graph $H$ such that $\TSsym_k(H)\cong\overline{K_n}$.
Let $V(H)=\{x_i^j:\ 1\le i\le n,\ 1\le j\le k\}$.
The graph $H$ is the union of $k$ cliques of size $n$: for each fixed $j$, make $\{x_1^j,\dots,x_n^j\}$ a clique.
Order the vertices of each clique by their first index, and connect two vertices in different cliques if and only if they are not in the same position in these two orderings.
Equivalently, for distinct $j,j'$, join $x_i^j$ to $x_{i'}^{j'}$ if and only if $i\ne i'$; see \cref{fig:edgeless-construction}.

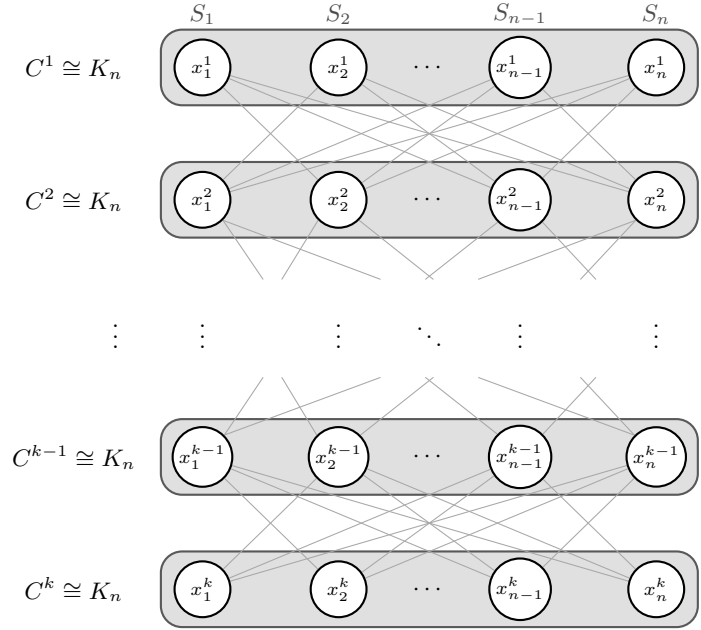
\begin{figure}[ht]
\centering
\begin{adjustbox}{max width=\textwidth}
\begin{tikzpicture}[
	vertex/.style={circle,draw,thick,fill=white,inner sep=1pt,minimum size=21pt,font=\scriptsize},
	guide/.style={fill=black!12,draw=black!65,line width=0.8pt,rounded corners=8pt},
	cross edge/.style={thin,black!34},
	half edge/.style={thin,black!34,line cap=round},
	note/.style={font=\small},
	rowlabel/.style={font=\small}
]
\def\rowleft{-0.55}
\def\rowright{6.55}
\def\rowhalfheight{0.50}
\def\dotsx{3.00}
\def\dotsy{1.35}

\foreach \r/\y/\jlab in {1/4.80/1,2/3.05/2,3/-0.35/{k-1},4/-2.10/k} {
	\begin{pgfonlayer}{background}
		\draw[guide] (\rowleft,\y-\rowhalfheight) rectangle (\rowright,\y+\rowhalfheight);
	\end{pgfonlayer}
	\node[rowlabel] at (-1.70,\y) {$C^{\jlab}\cong K_n$};
	\foreach \i/\ilab/\x in {1/1/0,2/2/1.80,3/{n-1}/4.20,4/n/6.00} {
		\node[vertex] (v-\r-\i) at (\x,\y) {$x_{\ilab}^{\jlab}$};
	}
	\node[note] at (\dotsx,\y) {$\cdots$};
}

\node[note] at (-1.15,\dotsy) {$\vdots$};
\foreach \x in {0,1.80,4.20,6.00} {
	\node[note] at (\x,\dotsy) {$\vdots$};
}
\node[note] at (\dotsx,\dotsy) {$\ddots$};

\begin{pgfonlayer}{background}
	\foreach \a/\b in {1/2,1/3,1/4,2/1,2/3,2/4,3/1,3/2,3/4,4/1,4/2,4/3} {
		\draw[cross edge] (v-1-\a) -- (v-2-\b);
		\draw[cross edge] (v-3-\a) -- (v-4-\b);
	}

	\draw[half edge] (v-2-1.south east) -- (0.80,2.00);
	\draw[half edge] (v-2-1.south east) -- (2.35,2.00);
	\draw[half edge] (v-2-2.south west) -- (1.05,2.00);
	\draw[half edge] (v-2-2.south east) -- (3.05,2.00);
	\draw[half edge] (v-2-3.south west) -- (2.95,2.00);
	\draw[half edge] (v-2-3.south east) -- (5.20,2.00);
	\draw[half edge] (v-2-4.south west) -- (3.65,2.00);
	\draw[half edge] (v-2-4.south west) -- (5.00,2.00);

	\draw[half edge] (v-3-1.north east) -- (0.80,0.70);
	\draw[half edge] (v-3-1.north east) -- (2.35,0.70);
	\draw[half edge] (v-3-2.north west) -- (1.05,0.70);
	\draw[half edge] (v-3-2.north east) -- (3.05,0.70);
	\draw[half edge] (v-3-3.north west) -- (2.95,0.70);
	\draw[half edge] (v-3-3.north east) -- (5.20,0.70);
	\draw[half edge] (v-3-4.north west) -- (3.65,0.70);
	\draw[half edge] (v-3-4.north west) -- (5.00,0.70);
\end{pgfonlayer}

\node[note,black!70] at (0,5.50) {$S_1$};
\node[note,black!70] at (1.80,5.50) {$S_2$};
\node[note,black!70] at (4.20,5.50) {$S_{n-1}$};
\node[note,black!70] at (6.00,5.50) {$S_n$};
\end{tikzpicture}
\end{adjustbox}
\caption{Host $H$ for \cref{thm:edgeless}; each shaded box is a clique.}
\label{fig:edgeless-construction}
\end{figure}

Every independent set of size $k$ in $H$ contains one vertex from each clique, and the cross-edges force all chosen vertices to have the same first index.
Thus, every independent $k$-set of $H$ is of the form $S_i:=\{x_i^1,\dots,x_i^k\}$ for some $1\le i\le n$.
For $i\ne i'$, the sets $S_i$ and $S_{i'}$ differ in all $k\ge 2$ positions, so they are not adjacent in $\TSsym_k(H)$.
Hence, the token-sliding graph with $k$ tokens of $H$ is an independent set of size $n$, and so $\TSsym_k(H)\cong\overline{K_n}$.
\end{proof}

Avis and Hoang realize the complete targets needed below.

\begin{theorem}[{\cite[Cor.~4]{AvisHoang2023}}]
\label{thm:complete-graphs-realizable}
For every fixed $k\ge 2$ and every $n\ge 2$, the complete graph $K_n$ is a $\TSsym_k$-graph.
\end{theorem}

We also use their path and cycle realizations.

\begin{theorem}[{\cite[Cor.~5]{AvisHoang2023}}]
\label{thm:paths-cycles}
The following hold.
\begin{enumerate}[label=\textup{(\alph*)},leftmargin=2em]
\item For every $n\ge 1$, $\TSsym_2(\overline{P_{n+1}})\cong L(P_{n+1})\cong P_n$.
\item For every fixed $k\ge 2$ and every $n\ge 3$, the cycle $C_n$ is a $\TSsym_k$-graph.
For $n\ge 4$, more precisely, $\TSsym_2(\overline{C_n})\cong C_n$.
\end{enumerate}
\end{theorem}

By the first item and \cref{prop:padding}, the path $P_n$ is a $\TSsym_k$-graph for every fixed $k\ge 2$.

We use their disjoint-union closure to assemble disconnected targets.

\begin{proposition}[{\cite[Prop.~10(a)]{AvisHoang2023}}]
\label{prop:target-disjoint-union}
Let $k\ge 2$.
If $F_1,\dots,F_p$ are $\TSsym_k$-graphs, then their disjoint union $F_1\sqcup\cdots\sqcup F_p$ is also a $\TSsym_k$-graph.
\end{proposition}

Avis and Hoang's results cover all targets on at most three vertices; we include the short proof for completeness.

\begin{corollary}
\label{cor:small-graphs-realizable}
For every fixed $k\ge 2$, every graph on at most three vertices is a $\TSsym_k$-graph.
\end{corollary}

\begin{proof}
We prove this by splitting the target into connected components already known to be realizable.
The empty graph on zero vertices is realized by $\TSsym_k(K_1)$, since $K_1$ has no independent $k$-set for $k\ge 2$ under the convention in \cref{sec:prelim}.
Every non-empty connected graph on at most three vertices is a path or a cycle.
Use the preceding path consequence and \cref{thm:paths-cycles} for connected components, and then use \cref{prop:target-disjoint-union}.
\end{proof}

The following two results use complements that split into complete and isolated components.

\begin{theorem}
\label{thm:complement-complete-bipartite}
For every fixed $k\ge 2$ and all $m,n\ge 1$, the complement $\overline{K_{m,n}}$ is a $\TSsym_k$-graph.
\end{theorem}

\begin{proof}
We show this by splitting the target into complete components.
The complement of $K_{m,n}$ is the disjoint union $K_m\sqcup K_n$.
For each $r\in\{m,n\}$, the graph $K_r$ is a $\TSsym_k$-graph: if $r=1$, this is the one-vertex edgeless graph and follows from \cref{thm:edgeless}; if $r\ge 2$, it follows from \cref{thm:complete-graphs-realizable}.
Therefore, $\overline{K_{m,n}}$ is a $\TSsym_k$-graph by \cref{prop:target-disjoint-union}.
\end{proof}

The complement of $K_n-e$ is one edge plus isolated vertices.

\begin{theorem}
\label{thm:complement-complete-minus-edge}
For every fixed $k\ge 2$, every $n\ge 2$, and every edge $e\in E(K_n)$, the complement $\overline{K_n-e}$ is a $\TSsym_k$-graph.
\end{theorem}

\begin{proof}
We show this by splitting off the one-edge component from the isolated vertices.
If $n=2$, then $\overline{K_2-e}\cong K_2$, which is a $\TSsym_k$-graph by \cref{thm:complete-graphs-realizable}.
Assume $n\ge 3$.
The only edge of $\overline{K_n-e}$ is the deleted edge~$e$, so $\overline{K_n-e}\cong K_2\sqcup (n-2)K_1$.
The component $K_2$ is a $\TSsym_k$-graph by \cref{thm:complete-graphs-realizable}, and each isolated vertex is a $\TSsym_k$-graph by \cref{thm:edgeless}.
Thus, $\overline{K_n-e}$ is a $\TSsym_k$-graph by \cref{prop:target-disjoint-union}.
\end{proof}

The following self-contained lemma is a diamond obstruction for token-sliding targets.
It shows that, in a target with no independent triple, an induced $K_4-e$ forces an additional independent target vertex.
Both hypotheses are used: the induced $K_4-e$ supplies the local diamond, while the absence of an independent triple rules out the extra token state forced by that diamond.

\begin{lemma}
\label{lem:diamond-independent-obstruction}
Let $k\ge 2$, and let $F$ contain an induced copy of $K_4-e$.
If $F$ has no independent set of size~$3$, then $F$ is not a $\TSsym_k$-graph.
\end{lemma}

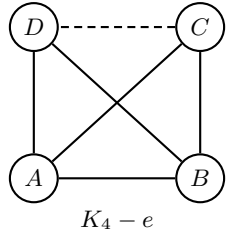
\begin{figure}[ht]
\centering
\begin{adjustbox}{max width=\textwidth}
\begin{tikzpicture}[
	vertex/.style={circle,draw,thick,inner sep=1pt,minimum size=18pt,font=\small},
	edge/.style={thick},
	missing/.style={thick,densely dashed},
	label/.style={font=\small}
]
\node[vertex] (A) at (0,0) {$A$};
\node[vertex] (B) at (2.2,0) {$B$};
\node[vertex] (C) at (2.2,2) {$C$};
\node[vertex] (D) at (0,2) {$D$};

\draw[edge] (A) -- (B);
\draw[edge] (A) -- (C);
\draw[edge] (A) -- (D);
\draw[edge] (B) -- (C);
\draw[edge] (B) -- (D);
\draw[missing] (C) -- (D);
\node[label] at (1.1,-0.55) {$K_4-e$};
\end{tikzpicture}
\end{adjustbox}
\caption{The obstruction for \cref{lem:diamond-independent-obstruction}; the dashed pair is the unique missing edge.}
\label{fig:k4-minus-edge}
\end{figure}

\begin{proof}
Suppose to the contrary that $F\cong\TSsym_k(G)$ for some graph~$G$.
We consider an induced $K_4-e$ subgraph of~$F$ with vertices $A,B,C,D$, where the naming is consistent with \cref{fig:k4-minus-edge}: the unique missing edge is~$CD$.
Identify $A,B,C,D$ with the corresponding independent $k$-sets of~$G$.

The common edge $AB$ fixes a common size-$(k-1)$ set.
Indeed, let $A=I\cup\{a\}$ and $B=I\cup\{b\}$, where $|I|=k-1$ and $ab\in E(G)$.
Any independent $k$-set adjacent to both $A$ and $B$ and distinct from them must have the form $I\cup\{z\}$.
Indeed, let $X$ be such a set.
From adjacency to $A$, either $X$ replaces $a$, giving $X=I\cup\{z\}$, or $X$ keeps $a$ and replaces some $x\in I$.
In the latter case, if the new vertex is $b$, then the set contains both $a$ and $b$, contradicting $ab\in E(G)$.
Otherwise, it differs from $B$ in two token positions and hence is not adjacent to $B$.
The case where $X$ keeps $b$ and replaces a token of $I$ is symmetric.
Therefore, for distinct vertices $c,d\notin I\cup\{a,b\}$, we have $C=I\cup\{c\}$ and $D=I\cup\{d\}$.
Since $CD$ is not an edge of $F$, we have $cd\notin E(G)$.
Also, because $C$ and $D$ are independent, neither $c$ nor $d$ is adjacent in $G$ to any vertex of~$I$.

Since $k\ge 2$, the set $I$ is nonempty.
Choose $x\in I$ and set $J:=(I\setminus\{x\})\cup\{c,d\}$.
The preceding paragraph shows that $J$ is also an independent $k$-set of~$G$ and therefore a vertex of~$F$.
Moreover, $J\ne C$ because $J$ contains $d$ while $C$ does not, and $J\ne D$ because $J$ contains $c$ while $D$ does not.
The vertex $J$ is not adjacent to $C$, because such an adjacency would require the slide edge $xd$, which does not exist.
Similarly, $J$ is not adjacent to $D$.
Thus, $C,D,J$ induce an independent set of size~$3$ in~$F$, a contradiction.
\end{proof}

We classify complements of paths and cycles.

\begin{theorem}
\label{thm:complement-paths}
For every $k\ge 2$ and every $n\ge 1$, the complement $\overline{P_n}$ is a $\TSsym_k$-graph if and only if $n\le 5$.
\end{theorem}

\begin{proof}
We prove the two directions separately.
\begin{itemize}
\item[$(\Leftarrow)$]
Assume $n\le 5$.
For $n\le 3$, the graph $\overline{P_n}$ has at most three vertices, so it is a $\TSsym_k$-graph by \cref{cor:small-graphs-realizable}.
For $n=4$, the graph $\overline{P_4}$ is isomorphic to $P_4$, and hence is a $\TSsym_k$-graph by the path consequence of \cref{thm:paths-cycles,prop:padding}.

It remains to handle $\overline{P_5}$.
Let $H$ have vertex set $\{1,2,3,4,5\}$ and edges $\{1,2\}$, $\{1,5\}$, $\{2,3\}$, $\{2,4\}$, and $\{3,4\}$.
The non-edges of $H$ are $\{1,3\}$, $\{1,4\}$, $\{2,5\}$, $\{3,5\}$, and $\{4,5\}$.
The graph formed by these non-edges is triangle-free, so $\alpha(H)=2$, and the independent $2$-sets are
\[
\{4,5\},\{1,3\},\{2,5\},\{1,4\},\{3,5\}.
\]
In the displayed linear order, two listed sets fail to be adjacent in $\TSsym_2(H)$ exactly when they are consecutive: either they are disjoint, or the exchanged vertices are not adjacent in~$H$.
Thus, $\TSsym_2(H)\cong\overline{P_5}$.
Since $\alpha(H)=2$, \cref{prop:padding} gives $\TSsym_k(H\sqcup (k-2)K_1)\cong\TSsym_2(H)\cong\overline{P_5}$ for every $k\ge 2$.

\item[$(\Rightarrow)$]
We prove the contrapositive.
Suppose $n\ge 6$.
The vertices $1,2,4,6$ induce exactly one edge in $P_n$, so they induce a $K_4-e$ in $\overline{P_n}$.
Since $P_n$ has clique number~$2$, the graph $\overline{P_n}$ has no independent set of size~$3$.
By \cref{lem:diamond-independent-obstruction}, $\overline{P_n}$ is not a $\TSsym_k$-graph.
\end{itemize}
\end{proof}

The cycle case has the same cutoff pattern, with a separate construction for $n=6$.

\begin{theorem}
\label{thm:complement-cycles}
For every $k\ge 2$ and every $n\ge 3$, the complement $\overline{C_n}$ is a $\TSsym_k$-graph if and only if $n\le 6$.
\end{theorem}

\begin{proof}
We prove the two directions separately.
\begin{itemize}
\item[$(\Leftarrow)$]
Assume $n\le 6$.
For $n=3$, the graph $\overline{C_3}$ is edgeless and is covered by \cref{thm:edgeless}.
For $n=4$, we have $\overline{C_4}\cong 2K_2$, which is a $\TSsym_k$-graph by \cref{thm:complete-graphs-realizable,prop:target-disjoint-union}.
For $n=5$, the cycle $C_5$ is self-complementary, so \cref{thm:paths-cycles} applies.

For $n=6$, we construct a graph $H$ such that $\TSsym_k(H)\cong\overline{C_6}$.
Take $H:=K_3\sqcup K_2\sqcup (k-2)K_1$.
The construction appears in \cref{fig:c6-realization}; the fixed isolated vertices are omitted from the state labels.
Thus every independent set of size $k$ in $H$ contains all $k-2$ isolated vertices, one vertex from the $K_3$, and one vertex from the $K_2$.
As there are three choices in $K_3$ and two choices in $K_2$, the graph $\TSsym_k(H)$ has six vertices.
One can only move the token in the $K_3$ or the token in the $K_2$; hence two such independent sets are adjacent exactly when one of these two choices changes and the other one is fixed.
Hence, $\TSsym_k(H)\cong K_3\square K_2$.
The graph $K_3\square K_2$ is the triangular prism.
If the states are labelled $ax,bx,cx,ay,by,cy$, then the complement contains the cycle
\[
ax,\,by,\,cx,\,ay,\,bx,\,cy,\,ax.
\]
The graph $K_3\square K_2$ has six vertices and nine edges, so its complement has exactly six edges.
The displayed cycle accounts for all six complement edges.
Thus the complement is $C_6$, and $K_3\square K_2\cong \overline{C_6}$.

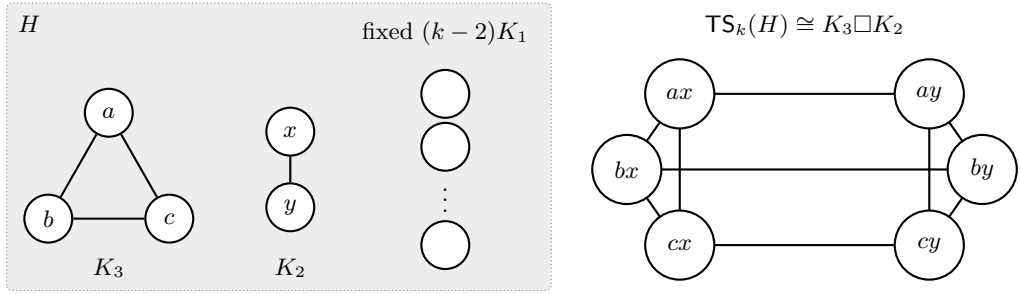
\begin{figure}[ht]
\centering
\begin{adjustbox}{max width=\textwidth}
\begin{tikzpicture}[
	vertex/.style={circle,draw,thick,fill=white,inner sep=1pt,minimum size=18pt,font=\small},
	state/.style={circle,draw,thick,fill=white,inner sep=1pt,minimum size=26pt,font=\small},
	edge/.style={thick},
	guide/.style={fill=black!7,draw=black!50,densely dotted,rounded corners=2pt},
	note/.style={font=\small},
	component/.style={font=\small}
]
\draw[guide] (-4.55,-0.95) rectangle (2.65,2.85);
\node[note,inner sep=1pt] at (-4.25,2.58) {$H$};

\node[vertex] (a) at (-3.2,1.4) {$a$};
\node[vertex] (b) at (-4.0,0) {$b$};
\node[vertex] (c) at (-2.4,0) {$c$};
\draw[edge] (a) -- (b) -- (c) -- (a);
\node[component] at (-3.2,-0.65) {$K_3$};

\node[vertex] (x) at (-0.8,1.15) {$x$};
\node[vertex] (y) at (-0.8,0.15) {$y$};
\draw[edge] (x) -- (y);
\node[component] at (-0.8,-0.65) {$K_2$};

\node[vertex] (f1) at (1.25,1.65) {};
\node[vertex] (f2) at (1.25,0.95) {};
\node[note] at (1.25,0.35) {$\vdots$};
\node[vertex] (f3) at (1.25,-0.35) {};
\node[component] at (1.25,2.45) {fixed $(k-2)K_1$};

\node[note] at (6.0,2.55) {$\TSsym_k(H)\cong K_3\square K_2$};
\coordinate (pax) at (4.35,1.65);
\coordinate (pbx) at (3.65,0.65);
\coordinate (pcx) at (4.35,-0.35);
\coordinate (pay) at (7.65,1.65);
\coordinate (pby) at (8.35,0.65);
\coordinate (pcy) at (7.65,-0.35);
\draw[edge] (pax) -- (pbx) -- (pcx) -- (pax);
\draw[edge] (pay) -- (pby) -- (pcy) -- (pay);
\draw[edge] (pax) -- (pay);
\draw[edge] (pbx) -- (pby);
\draw[edge] (pcx) -- (pcy);
\node[state] at (pax) {$ax$};
\node[state] at (pbx) {$bx$};
\node[state] at (pcx) {$cx$};
\node[state] at (pay) {$ay$};
\node[state] at (pby) {$by$};
\node[state] at (pcy) {$cy$};
\end{tikzpicture}
\end{adjustbox}
\caption{The $n=6$ construction for \cref{thm:complement-cycles}: $H:=K_3\sqcup K_2\sqcup (k-2)K_1$ and $\TSsym_k(H)\cong K_3\square K_2\cong\overline{C_6}$. Fixed isolated vertices are omitted from the state labels.}
\label{fig:c6-realization}
\end{figure}

\item[$(\Rightarrow)$]
We again prove the contrapositive.
Suppose $n\ge 7$.
The vertices $1,2,4,6$ induce exactly one edge in $C_n$, so they induce a $K_4-e$ in $\overline{C_n}$.
Since $C_n$ has clique number~$2$, the graph $\overline{C_n}$ has no independent set of size~$3$.
By \cref{lem:diamond-independent-obstruction}, $\overline{C_n}$ is not a $\TSsym_k$-graph.
\end{itemize}
\end{proof}

\section{Products of Some Basic Graphs}
\label{sec:product-construction}

This section turns disjoint unions in the host graph into Cartesian products in the token graph.
We prove the product formula and full-token case, apply them to basic Cartesian-product families and cycle--path products, and record the first two-token product obstructions.
The positive results use \cref{thm:product-lemma,cor:product-saturation} together with the path, cycle, complete-graph, and padding constructions from the previous section.
The negative two-token results use complement roots, the local degree formula of \cref{prop:per-edge-degree}, and induced four-cycle counts.

\begin{theorem}
\label{thm:product-lemma}
Let $G_1,G_2$ be graphs with disjoint vertex sets, let $G:=G_1\sqcup G_2$, and let $k\ge 0$ be an integer.
Then
\[
\TSsym_k(G)\;\cong\;\bigsqcup_{i=0}^{k}\big(\TSsym_{k-i}(G_1)\,\square\,\TSsym_i(G_2)\big),
\]
where the right-hand side is a disjoint union of Cartesian products and a factor is taken to be the empty graph on zero vertices whenever the requested token number exceeds the corresponding component's independence number.
\end{theorem}

\begin{proof}
Let $H$ be the graph on the right-hand side.
We verify that $\TSsym_k(G)\cong H$ by using the natural component map.
Every independent $k$-set $S$ of $G$ decomposes uniquely as $S=S_1\sqcup S_2$, where $S_1:=S\cap V(G_1)$ and $S_2:=S\cap V(G_2)$.
Both sets are independent in their respective components, and $|S_1|+|S_2|=k$.
Let $i:=|S_2|$.
The map
\[
\varphi\colon V(\TSsym_k(G))\to V(H)
\]
that sends $S$ to $(S_1,S_2)$ in the $i$-th summand is a bijection.

Observe that any token-slide in $G$ can only be along an edge of either $G_1$ or~$G_2$.
Thus, two size-$k$ independent sets $S,S'$ of $G$ differ by one token-slide if and only if exactly one of $\{S_1,S_1'\}$ and $\{S_2,S_2'\}$ differs by one token-slide and the other is fixed.
This means precisely that $\varphi(S)$ and $\varphi(S')$ lie in the same summand and are adjacent in the Cartesian product $\TSsym_{k-i}(G_1)\square \TSsym_i(G_2)$.
No edge of $\TSsym_k(G)$ crosses two summands.
Therefore, $\varphi$ is a graph isomorphism.
\end{proof}

When all tokens must be used, only one Cartesian-product summand remains.

\begin{corollary}
\label{cor:product-saturation}
Let $G_1,G_2$ be graphs with $\alpha(G_1)=a\ge 1$ and $\alpha(G_2)=b\ge 1$.
Then $\TSsym_{a+b}(G_1\sqcup G_2)\cong \TSsym_a(G_1)\square \TSsym_b(G_2)$.
\end{corollary}

\begin{proof}
To prove the full-token case, we observe that all summands but one are empty.
In the formula of \cref{thm:product-lemma} with $k=a+b$, the factor $\TSsym_{k-i}(G_1)$ is empty for $k-i>a$ (i.e.\ $i<b$) and $\TSsym_i(G_2)$ is empty for $i>b$, so the unique non-empty summand is $i=b$, giving $\TSsym_a(G_1)\square \TSsym_b(G_2)$.
\end{proof}

The product formula gives several product families.

\begin{corollary}
\label{cor:product-families}
For a fixed $k\ge 2$, the following Cartesian products are $\TSsym_k$-graphs.
\begin{enumerate}[label=\textup{(\alph*)}]
\item $K_n\square K_m$ for all $m,n\ge 1$.
\item If $k\ge 3$, then $P_n\square K_m$ for all $m,n\ge 1$.
\item If $k\ge 3$, then $C_n\square K_m$ for all $m\ge 1$ and all $n\ge 3$.
\item If $k\ge 4$, then $C_m\square C_n$ for all $m,n\ge 3$.
\end{enumerate}
\end{corollary}

\begin{proof}
We apply the full-token product formula to the known realizations.
\begin{enumerate}[label=\textup{(\alph*)},leftmargin=2em]
\item Take $A:=K_n\sqcup (k-2)K_1$ and $B:=K_m$.
Observe that $\alpha(A)=k-1$ and $\alpha(B)=1$.
By \cref{prop:padding}, $\TSsym_{k-1}(A)\cong K_n$.
Since $\alpha(B)=1$, the $B$-factor in \cref{cor:product-saturation} is the graph induced by the singleton independent sets of~$B$, namely~$K_m$.

\item Take $A:=\overline{P_{n+1}}\sqcup (k-3)K_1$ and $B:=K_m$.
Observe that $\alpha(A)=k-1$ and $\alpha(B)=1$.
By \cref{thm:paths-cycles,prop:padding}, $\TSsym_{k-1}(A)\cong P_n$, and the same $B$-factor is $K_m$.

\item When $n=3$, the claim follows from \textup{(a)} because $C_3\cong K_3$.
When $n\ge 4$, take $A:=\overline{C_n}\sqcup (k-3)K_1$ and $B:=K_m$.
Observe that $\alpha(A)=k-1$ and $\alpha(B)=1$.
By \cref{thm:paths-cycles,prop:padding}, $\TSsym_{k-1}(A)\cong C_n$, and \cref{cor:product-saturation} gives $C_n\square K_m$.

\item If one of $m,n$ is equal to $3$, then the claim follows from \textup{(a)} or \textup{(c)}, since $C_3\cong K_3$.
It remains to consider $m,n\ge 4$.
Set $A:=\overline{C_m}$ and $B:=\overline{C_n}$.
Since $C_m$ and $C_n$ are triangle-free cycles, $\alpha(A)=\alpha(B)=2$.
By \cref{thm:paths-cycles}, $\TSsym_2(A)\cong C_m$ and $\TSsym_2(B)\cong C_n$.
Thus, \cref{cor:product-saturation} gives $\TSsym_4(A\sqcup B)\cong C_m\square C_n$.
For $k\ge 4$, add $k-4$ isolated vertices to $A\sqcup B$ and apply \cref{prop:padding}.
\end{enumerate}
\end{proof}

The same method gives positive cycle--path cases.

\begin{corollary}
\label{cor:cycle-path-positive}
Let $m\ge 3$, $n\ge 1$, and $k\ge 2$.
The graph $C_m\square P_n$ is a $\TSsym_k$-graph whenever one of the following holds:
\begin{enumerate}[label=\textup{(\alph*)}]
\item $k=2$ and either $n=1$ or $(m,n)=(3,2)$;
\item $k=3$ and either $m=3$ or $n\le 2$;
\item $k\ge 4$.
\end{enumerate}
\end{corollary}

\begin{proof}
We prove the listed cases by reducing them to known cycle, complete-graph, and product cases.
\begin{enumerate}[label=\textup{(\alph*)},leftmargin=2em]
\item Suppose $k=2$.
If $n=1$, then $C_m\square P_1\cong C_m$.
For $m\ge 4$, use $\TSsym_2(\overline{C_m})\cong C_m$ from \cref{thm:paths-cycles}.
For $m=3$, use $C_3\cong K_3\cong \TSsym_2(K_3\sqcup K_1)$.
The remaining case is $(m,n)=(3,2)$.
Since every independent $2$-set of $K_3\sqcup K_2$ has one vertex in each component, a slide changes exactly one coordinate.
Thus $\TSsym_2(K_3\sqcup K_2)\cong K_3\square K_2\cong C_3\square P_2$.

\item Suppose $k=3$.
If $m=3$, then $C_m\cong K_3$.
Apply \cref{cor:product-saturation} to $A:=\overline{P_{n+1}}$ and $B:=K_3$.
Here $\alpha(A)=2$, $\alpha(B)=1$, $\TSsym_2(A)\cong P_n$ by \cref{thm:paths-cycles}, and $\TSsym_1(B)\cong K_3$.
Thus, $\TSsym_3(A\sqcup B)\cong P_n\square K_3\cong C_3\square P_n$.
Hence assume $m\ge 4$.
If $n=1$, then $C_m\square P_1\cong C_m$, and the $k=2$ constructions above can be padded by one isolated vertex.
If $n=2$, apply \cref{cor:product-saturation} to $A:=\overline{C_m}$ and $B:=K_2$.
Then $\alpha(A)=2$, $\alpha(B)=1$, $\TSsym_2(A)\cong C_m$, and $\TSsym_1(B)\cong K_2\cong P_2$.
So $\TSsym_3(A\sqcup B)\cong C_m\square P_2$.

\item Suppose $k\ge 4$.
If $n=1$, then $C_m\square P_1\cong C_m$, and the $k=2$ constructions in part~\textup{(a)} can be padded by $k-2$ isolated vertices.
Hence assume $n\ge 2$.
If $m=3$, use the construction from part~\textup{(b)} and pad by $k-3$ isolated vertices.
Hence assume $m\ge 4$.
If $n=2$, use the $n=2$ construction from part~\textup{(b)} and pad by $k-3$ isolated vertices.
It remains that $m\ge 4$ and $n\ge 3$.
Set $A:=\overline{C_m}$ and $B:=\overline{P_{n+1}}$.
Then $\alpha(A)=\alpha(B)=2$, $\TSsym_2(A)\cong C_m$, and $\TSsym_2(B)\cong P_n$.
Hence, \cref{cor:product-saturation} gives $\TSsym_4(A\sqcup B)\cong C_m\square P_n$.
Padding gives the result for every $k\ge 4$.
\end{enumerate}
\end{proof}

Some two-token product cases remain.
The next two lemmas rule out the prism cases and the first three triangular strips; the later obstruction for $C_4\square C_n$ uses induced $4$-cycle counts rather than layer interfaces.
\Cref{lem:cycle-path-prism-k2} forces a reduced complement root to be bipartite and claw-free, hence essentially a path or cycle.
\Cref{lem:cycle-path-triangular-strip-k2} tracks the possible interfaces between consecutive triangular layers.

\begin{lemma}
\label{lem:cycle-path-prism-k2}
For every $m\ge 4$, the graph $C_m\square P_2$ is not a $\TSsym_2$-graph.
\end{lemma}

\begin{proof}
We show non-realizability by moving to a claw-free complement root and using degree parity.
Suppose to the contrary that $C_m\square P_2\cong \TSsym_2(G)$.
The target graph is connected. By \cref{lem:sizebound}, there is a graph $G'$ such that $\TSsym_2(G')\cong \TSsym_2(G)$ and, putting $F:=\overline{G'}$, the graph $F$ is connected, has no isolated vertices, and $\TSsym_2(\overline{F})\cong C_m\square P_2$.

For $m\ge 4$, the graph $C_m\square P_2$ is triangle-free and $3$-regular.
Since $\TSsym_2(\overline{F})$ is triangle-free, \cref{prop:claw-iff-triangle-free} implies that $F$ is claw-free.

Let $ab\in E(F)$.
The corresponding vertex of $\TSsym_2(\overline{F})$ has degree $3$, so \cref{prop:per-edge-degree} gives $3=\deg_F(a)+\deg_F(b)-2-2|N_F(a)\cap N_F(b)|$.
Thus, $\deg_F(a)+\deg_F(b)$ is odd for every edge $ab\in E(F)$.
Hence, every edge of $F$ joins vertices of opposite degree parity, and degree parity gives a proper $2$-coloring of~$F$.
So $F$ is bipartite.

A bipartite claw-free graph has maximum degree at most $2$: if a vertex had three neighbors, those neighbors would be pairwise non-adjacent and would form an induced claw.
Since $F$ is connected and has no isolated vertices, $F$ is a path or a cycle.
Since $F$ is bipartite, it is triangle-free.
Therefore, \cref{lem:line-avis} gives $\TSsym_2(\overline{F})\cong L(F)$.
But the line graph of a path or a cycle has maximum degree at most $2$, whereas $C_m\square P_2$ is $3$-regular.
The contradiction proves the result.
\end{proof}

The second lemma treats the small triangular strips through the triangular layers of $K_3\square P_n$.
\Cref{fig:triangular-strip-interfaces} shows the two interface types used in the proof and the blockers forced by a same-center interface.

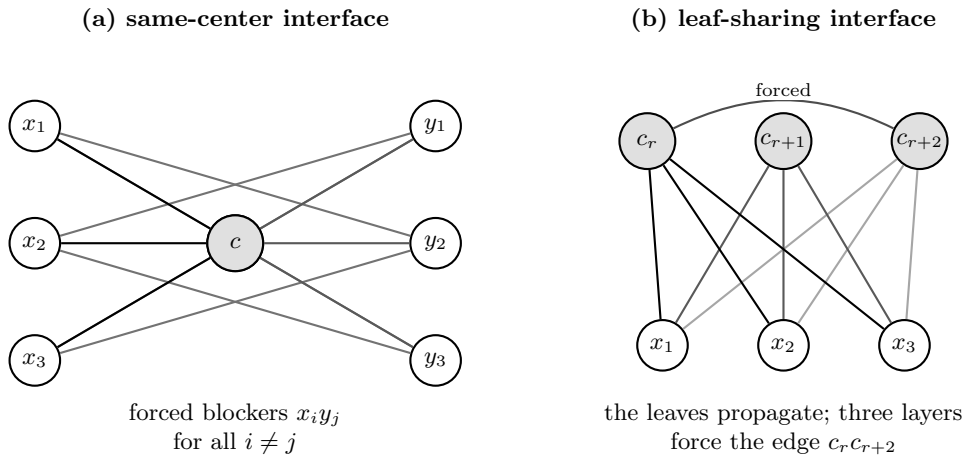
\begin{figure}[ht]
\centering
\begin{adjustbox}{max width=\textwidth}
\begin{tikzpicture}[
	rootv/.style={circle,draw,thick,fill=white,inner sep=1pt,minimum size=19pt,font=\small},
	center/.style={circle,draw,thick,fill=black!12,inner sep=1pt,minimum size=21pt,font=\small},
	rootedge/.style={thick},
	second/.style={thick,draw=black!65},
	third/.style={thick,draw=black!35},
	blocker/.style={thick,draw=black!55},
	forced/.style={thick,draw=black!70},
	panel/.style={font=\small\bfseries},
	note/.style={font=\small,align=center},
	label/.style={font=\scriptsize,inner sep=1pt,fill=white}
]
\node[panel] at (0,2.95) {(a) same-center interface};
\node[center] (c) at (0,0) {$c$};
\foreach \i/\y in {1/1.55,2/0,3/-1.55} {
	\node[rootv] (x\i) at (-2.65,\y) {$x_{\i}$};
	\node[rootv] (y\i) at (2.65,\y) {$y_{\i}$};
}
\draw[blocker] (x1) -- (y2);
\draw[blocker] (x1) -- (y3);
\draw[blocker] (x2) -- (y1);
\draw[blocker] (x2) -- (y3);
\draw[blocker] (x3) -- (y1);
\draw[blocker] (x3) -- (y2);
\foreach \i in {1,2,3} {
	\draw[rootedge] (c) -- (x\i);
	\draw[second] (c) -- (y\i);
}
\node[center] at (c) {$c$};
\node[note] at (0,-2.45) {forced blockers $x_i y_j$\\for all $i\ne j$};

\node[panel] at (7.25,2.95) {(b) leaf-sharing interface};
\foreach \i/\x in {1/5.65,2/7.25,3/8.85} {
	\node[rootv] (z\i) at (\x,-1.35) {$x_{\i}$};
}
\node[center] (cr) at (5.45,1.35) {$c_r$};
\node[center] (cs) at (7.25,1.35) {$c_{r+1}$};
\node[center] (ct) at (9.05,1.35) {$c_{r+2}$};
\foreach \i in {1,2,3} {
	\draw[rootedge] (cr) -- (z\i);
	\draw[second] (cs) -- (z\i);
	\draw[third] (ct) -- (z\i);
}
\draw[forced] (cr) to[bend left=26] node[label,above] {forced} (ct);
\node[note] at (7.25,-2.45) {the leaves propagate; three layers\\force the edge $c_r c_{r+2}$};
\end{tikzpicture}
\end{adjustbox}
\caption{Interfaces between consecutive triangular layers in \cref{lem:cycle-path-triangular-strip-k2}. Same-center interfaces force the blocker edges $x_i y_j$ with $i\ne j$; three leaf-sharing layers force $c_r c_{r+2}$.}
\label{fig:triangular-strip-interfaces}
\end{figure}

\begin{lemma}
\label{lem:cycle-path-triangular-strip-k2}
For $3\le n\le 5$, the graph $C_3\square P_n$ is not a $\TSsym_2$-graph.
\end{lemma}

\begin{proof}
We prove the exclusion by following how the three edges in each triangular layer can meet in the complement root.
Since $C_3=K_3$, suppose to the contrary that $K_3\square P_n\cong \TSsym_2(G)$.
The target is connected and has $3n$ vertices.
By \cref{lem:sizebound}, we may replace $G$ by a graph whose complement root $F$ satisfies $\TSsym_2(\overline{F})\cong K_3\square P_n$ and $|E(F)|=3n$.
Fix an isomorphism $\phi:\TSsym_2(\overline{F})\to K_3\square P_n$.
We identify the vertices of $\TSsym_2(\overline{F})$ with the edges of~$F$.
We shall use the following rule: if two edges $ab,ac\in E(F)$ share the endpoint $a$, then they are adjacent in $\TSsym_2(\overline{F})$ if and only if $bc\notin E(F)$; disjoint edges of $F$ are not adjacent.

Label the vertices of $K_3\square P_n$ as $(i,r)$, where $i\in\{1,2,3\}$ and $1\le r\le n$.
Let $L_r:=\{(1,r),(2,r),(3,r)\}$ and $\widehat L_r:=\phi^{-1}(L_r)\subseteq E(F)$.
The sets $\widehat L_1,\ldots,\widehat L_n$ partition $E(F)$.
Each $L_r$ is a triangle, so the three edges in $\widehat L_r$ are pairwise adjacent in $\TSsym_2(\overline{F})$.
Three pairwise incident edges in a simple graph form either a triangle or a star.
They cannot form a triangle in~$F$, since two edges of a root triangle have adjacent non-shared endpoints and hence are not adjacent in $\TSsym_2(\overline{F})$.
Thus, each $\widehat L_r$ is an induced \emph{root claw} $c_r x_{r,1}$, $c_r x_{r,2}$, $c_r x_{r,3}$,
where the leaves $x_{r,1},x_{r,2},x_{r,3}$ are pairwise non-adjacent in~$F$.

Consider two consecutive target layers.
Let their root claws be $cx_1,cx_2,cx_3$ and $dy_1,dy_2,dy_3$,
indexed so that the matching target adjacencies are $cx_i\sim dy_i$.
Each matched pair must share an endpoint.
For a matched pair $cx_i,dy_i$, the only possible endpoint sharings are $c=d$, $c=y_i$, $x_i=d$, or $x_i=y_i$.
A mixed center-leaf sharing is impossible.
Indeed, suppose $c=y_i$.
Then $d=c$ would make $dy_i$ a loop, so $d\ne c$, and $dc$ is an edge of the second claw.
For $j\ne i$, the matched pair $cx_j,dy_j$ must share an endpoint.
It cannot have $c=y_j$ or $x_j=d$, since either equality would repeat the edge $dc$.
It cannot have $x_j=y_j$, because then the non-shared endpoints $c,d$ are adjacent in~$F$, so the two root edges are not adjacent in $\TSsym_2(\overline{F})$.
The case $x_i=d$ is symmetric.
After mixed center-leaf sharing is excluded, only the possibilities $c=d$ and $x_i=y_i$ remain.
If $c=d$, then the two claws have the same center; if $c\ne d$, then every matched pair must share its leaf.
Hence, every adjacent pair of layers has one of two \emph{interface types}: either the two claws have the same center (\emph{same-center interface}), or their matched leaves are equal (\emph{leaf-sharing interface}).

\begin{itemize}
\item \textbf{Case 1: Some consecutive pair of layers has a same-center interface.}
Suppose two consecutive layers are represented by $cx_1,cx_2,cx_3$ and $cy_1,cy_2,cy_3$.
Since the two target layers are disjoint, the six leaves
$x_1,x_2,x_3,y_1,y_2,y_3$ are distinct.
For every $i\ne j$, the target vertices $cx_i$ and $cy_j$ are non-adjacent but share the endpoint~$c$.
By the rule above, $x_i y_j\in E(F)$.
Thus, a same-center interface forces the six distinct blocker edges $\{x_i y_j:\ i\ne j\}$.

No such blocker edge is adjacent in $\TSsym_2(\overline{F})$ to any vertex of the two interface layers.
Indeed, $x_i y_j$ can only meet the interface at $cx_i$ or $cy_j$.
Its possible adjacency to $cx_i$ is blocked by the edge $cy_j$, and its possible adjacency to $cy_j$ is blocked by the edge $cx_i$.
All other pairs are disjoint.
Therefore, a blocker cannot lie in either interface layer.
It also cannot lie in a target layer adjacent to one of them, since each vertex in such a layer has its matching target neighbor in the interface layer.
Since the target layers partition the root edges via the fixed isomorphism, all blocker edges must therefore be represented by vertices in the remaining far layers.

For $n\le 5$, relative to any fixed consecutive pair of layers, at most two target layers are neither the pair itself nor adjacent to either member of the pair.
A single target triangle layer contains at most two of the six blockers: the blockers form $K_{3,3}$ minus a perfect matching on
$\{x_1,x_2,x_3\}\cup\{y_1,y_2,y_3\}$,
so no three blocker edges have a common endpoint, whereas every target triangle layer is represented by a root claw.
Thus three blockers cannot all be represented in one target layer, because the three root edges in such a layer must share a single center.
Thus, the remaining far layers can contain at most four blockers.
This upper bound contradicts the six blockers forced above.
Therefore, no same-center interface occurs.

\item \textbf{Case 2: Every consecutive pair of layers has a leaf-sharing interface.}
Since Case~1 is impossible, every consecutive pair of layers has a leaf-sharing interface.
Consecutive target layers are joined by the matching edges $(i,r)(i,r+1)$, so the shared leaves propagate along the path.
After choosing the labels in the first layer, all layers have the form $c_r x_1,c_r x_2,c_r x_3$ for $1\le r\le n$.
The centers $c_r$ are distinct, and no center is one of the common leaves; otherwise some layer edge would be repeated or would be a loop.

Now, take three consecutive layers $r,r+1,r+2$.
For each $i$, the target vertices corresponding to $c_r x_i$ and $c_{r+2}x_i$ are non-adjacent but share the root endpoint $x_i$.
Hence, $c_r c_{r+2}\in E(F)$.
The edge joins two centers, so it is not one of the layer-claw edges $c_sx_i$.
But the layer-claw edges are $3n$ distinct edges of $F$, and $|E(F)|=3n$.
They therefore exhaust $E(F)$, so the edge $c_r c_{r+2}$ cannot exist.
\end{itemize}
The contradiction proves the lemma.
\end{proof}

Together, the two lemmas give the following exclusions.

\begin{proposition}
\label{prop:cycle-path-k2-structural}
Let $m\ge 3$ and $n\ge 1$.
If either $m\ge 4$ and $n=2$, or $m=3$ and $3\le n\le 5$, then $C_m\square P_n$ is not a $\TSsym_2$-graph.
\end{proposition}

\begin{proof}
We prove the proposition by using the two preceding obstruction lemmas.
\begin{itemize}
\item If $m\ge 4$ and $n=2$, the result is \cref{lem:cycle-path-prism-k2}.

\item If $m=3$ and $3\le n\le 5$, the result is \cref{lem:cycle-path-triangular-strip-k2}.
\end{itemize}
\end{proof}

\begin{remark}
\label{rem:cycle-path-frontier}
Combining \cref{cor:cycle-path-positive,prop:cycle-path-k2-structural}, the unsettled two-token cycle--path cases are $(3,n)$ with $n\ge 6$, and $(m,n)$ with $m\ge 4$ and $n\ge 3$.
Among these, $C_4\square P_3$ has the fewest vertices.
At three tokens, the remaining open cycle--path cases are $m\ge 4$ and $n\ge 3$.
\end{remark}

Corollary~\ref{cor:product-families}\textup{(d)} realizes every cycle product $C_m\square C_n$ with $k\ge 4$.
This positive result does not extend to $k=2$: already the products $C_4\square C_n$ are obstructed.

We use the following notation.
If $F$ is a graph and $H=\TSsym_2(\overline{F})$, then the vertices of $H$ are the edges of $F$.
Two edges $xy,xz\in E(F)$ are adjacent in $H$ exactly when $yz\notin E(F)$.
When $X,Y\subseteq V(F)$ are disjoint, let $m_F(X,Y)$ denote the number of edges of $F$ with one endpoint in $X$ and one endpoint in~$Y$.
We also use a simple consequence of triangle-freeness.
If $H=\TSsym_2(\overline{F})$ is triangle-free, then $\alpha(F[N_F(v)])\le 2$ for every $v\in V(F)$.
Indeed, if $x,y,z\in N_F(v)$ were pairwise non-adjacent in $F$, then the three edges $vx,vy,vz$ would be pairwise adjacent in~$H$.

The proof of \cref{thm:c4-cycle-products-not-ts2} has three steps.
First, \cref{lem:ts2-local-c4-count} expresses the number of induced $4$-cycles through a vertex $ab$ of $\TSsym_2(\overline{F})$ using the local sets $A$, $B$, and $C$ around the edge $ab$ in the complement root.
Second, \cref{lem:ts2-four-regular-c4-cap} shows that triangle-free $4$-regular two-token graphs have at most four such cycles through each vertex.
Third, \cref{lem:cycle-product-local-c4-count} shows that the target $C_4\square C_n$ has at least five.

Induced $4$-cycles in $H$ can be counted from the local data around the corresponding edge of~$F$.

\begin{lemma}
\label{lem:ts2-local-c4-count}
Let $F$ be a graph, let $H:=\TSsym_2(\overline{F})$, and suppose that $H$ is triangle-free.
Fix an edge $ab\in E(F)$, viewed as a vertex of $H$, and put $A:=N_F(a)\setminus N_F[b]$, $B:=N_F(b)\setminus N_F[a]$, and $C:=N_F(a)\cap N_F(b)$.
Let $\mathcal C_H(ab)$ be the set of induced $4$-cycles of $H$ that contain the vertex~$ab$.
Then
\[
|\mathcal C_H(ab)|
=m_F(A,B)
+\sum_{t\in C}\binom{|A\setminus N_F(t)|}{2}
+\sum_{t\in C}\binom{|B\setminus N_F(t)|}{2}.
\]
\end{lemma}

\begin{proof}
We prove the formula by first listing the two neighbors of~$ab$ that can appear in an induced four-cycle.
When those neighbors lie on opposite sides of the edge, they determine the fourth vertex.
When they lie on the same side, the fourth vertex is supplied by a common witness in~$C$.

The neighbors of $ab$ in $H$ are exactly $\{au:u\in A\}\cup\{bv:v\in B\}$.
Since $H$ is triangle-free, both $A$ and $B$ are cliques in $F$: if, for example, distinct $u,v\in A$ were non-adjacent in $F$, then $ab,au,av$ would form a triangle in~$H$.

Let $Q$ be an induced $4$-cycle of $H$ containing~$ab$.
The two neighbors of $ab$ on $Q$ are either one vertex of the form $au$ and one of the form $bv$, or two vertices on the same side.

\begin{itemize}
\item If the two neighbors lie on opposite sides, say $au$ and $bv$, then the fourth vertex of $Q$ must be~$uv$.
Indeed, a common neighbor of $au$ and $bv$ must share one endpoint with each of these two root edges.
The candidate $ab$ is the old vertex, while $av$ and $bu$ are not edges of~$F$ by the definitions of $A$ and~$B$.
Thus only $uv$ can supply the fourth vertex.
The set $uv$ is a vertex of $H$ exactly when $uv\in E(F)$, and then $ab,au,uv,bv$ is an induced $4$-cycle of~$H$.
These cycles are counted by $m_F(A,B)$.

\item Suppose the two neighbors of $ab$ on $Q$ lie on the same side, say $au$ and $av$ with $u,v\in A$.
Since $A$ is a clique, $au$ and $av$ are not adjacent in~$H$.
A common neighbor of $au$ and $av$ must either be $uv$ or have the form~$at$.
The vertex $uv$ is not adjacent to either $au$ or $av$, because $av,au\in E(F)$.
Hence, the fourth vertex has the form $at$, where $t\in C$.
The condition $t\in C$ makes $at$ non-adjacent to $ab$.
The conditions $tu,tv\notin E(F)$ give the two token-sliding adjacencies $at\sim au$ and $at\sim av$.
Thus, the same-side cycles on the $a$-side are counted by $\displaystyle \sum_{t\in C}\binom{|A\setminus N_F(t)|}{2}$.
The $b$-side is symmetric.
\end{itemize}
These three families are disjoint and exhaust all induced $4$-cycles through~$ab$, proving the formula.
\end{proof}

The local formula gives a common upper bound when the two-token graph is triangle-free and $4$-regular.

\begin{lemma}
\label{lem:ts2-four-regular-c4-cap}
Let $F$ be a graph and let $H:=\TSsym_2(\overline{F})$.
If $H$ is triangle-free and $4$-regular, then every vertex of $H$ lies on at most four induced $4$-cycles.
\end{lemma}

\begin{proof}
We prove the bound by applying the local four-cycle count with the degree limits from $4$-regularity.
Once the two sides of the edge are denoted by $A$ and $B$, only a few sizes of $A$ and $B$ are possible, and each gives a bound of four.
The cases are organized by the sizes $x=|A|$ and $y=|B|$.
The $(3,1)$ and $(1,3)$ cases are immediate from the side-sum bound.
The one-sided $(4,0)$ and $(0,4)$ cases contain the main counting work, where the possibilities $|C|=6$ and $|C|=5$ are excluded.
The balanced $(2,2)$ case is controlled by a final common-neighbor count.

Fix a vertex $ab\in V(H)$, equivalently an edge $ab\in E(F)$, and use the notation $A,B,C$ from \cref{lem:ts2-local-c4-count}.
Since $H$ is $4$-regular, $|A|+|B|=4$.
For $t\in C$, set $x:=|A|$, $y:=|B|$, $p_A(t):=|A\setminus N_F(t)|$, and $p_B(t):=|B\setminus N_F(t)|$.

The following local degree identity, which is \cref{prop:per-edge-degree} in the present notation, will be used repeatedly.
For every edge $rs\in E(F)$,
\begin{equation}
\label{eq:product-local-edge-degree}
\deg_H(rs)=\deg_F(r)+\deg_F(s)-2-2|N_F(r)\cap N_F(s)|.
\end{equation}
Indeed, the neighbors of $rs$ in $H$ are the edges $rw$ with $w\in N_F(r)\setminus N_F[s]$, together with the edges $sw$ with $w\in N_F(s)\setminus N_F[r]$.

For $t\in C$, subtracting \eqref{eq:product-local-edge-degree} for the two edges $at$ and $bt$ gives
\begin{equation}
\label{eq:product-pa-pb}
p_A(t)-p_B(t)=x-2.
\end{equation}
Here $\deg_F(a)=1+x+|C|$ and $\deg_F(b)=1+y+|C|$.
The common neighbors of $a,t$ are $b$, the vertices of $A\cap N_F(t)$, and the vertices of $C\cap N_F(t)$.
The common neighbors of $b,t$ are $a$, the vertices of $B\cap N_F(t)$, and the vertices of $C\cap N_F(t)$.

For each $u\in A$, the vertex $au$ of $H$ is adjacent to $ab$, to each $at$ with $t\in C\setminus N_F(u)$, and to each $uv$ with $v\in B\cap N_F(u)$.
These are distinct neighbors, and $H$ is $4$-regular; hence $|C\setminus N_F(u)|+|B\cap N_F(u)|\le 3$.
Summing over $u\in A$ yields $\displaystyle m_F(A,B)+\sum_{t\in C}p_A(t)\le 3x$.
Symmetrically,
\begin{equation}
\label{eq:product-b-side-sum}
m_F(A,B)+\sum_{t\in C}p_B(t)\le 3y.
\end{equation}

By \cref{lem:ts2-local-c4-count}, the number of induced $4$-cycles of $H$ through $ab$ is
\begin{equation}
\label{eq:product-q-count}
Q:=m_F(A,B)+\sum_{t\in C}\binom{p_A(t)}2
+\sum_{t\in C}\binom{p_B(t)}2.
\end{equation}
We prove $Q\le 4$ by cases on $(x,y)$.

\begin{itemize}
\item \textbf{Case 1: $(x,y)\in\{(3,1),(1,3)\}$.}
By symmetry, it suffices to treat $(x,y)=(3,1)$.
In this case, \eqref{eq:product-pa-pb} gives $(p_A(t),p_B(t))=(1,0)$ or $(2,1)$.
Hence, the summand $\displaystyle \binom{p_A(t)}2+\binom{p_B(t)}2$ is $p_B(t)$, and \eqref{eq:product-b-side-sum} gives $Q\le 3$.

\item \textbf{Case 2: $(x,y)\in\{(4,0),(0,4)\}$.}
By symmetry, it suffices to treat $(x,y)=(4,0)$.
Then \eqref{eq:product-pa-pb} gives $p_A(t)=2$ for every $t\in C$, so $Q=|C|$.
Let $A=\{u_1,u_2,u_3,u_4\}$, set $c:=|C|$, and set $r_i:=|N_F(u_i)\cap C|$.
Since every vertex of $C$ is adjacent to exactly two vertices of $A$,
\begin{equation}
\label{eq:product-one-sided-rsum}
\sum_{i=1}^4 r_i=2c.
\end{equation}
Applying \eqref{eq:product-local-edge-degree} to $au_i$ gives
\begin{equation}
\label{eq:product-one-sided-deg}
\deg_F(u_i)=7-c+2r_i.
\end{equation}
For $i<j$, let $\lambda_{ij}$ be the number of vertices of $C$ adjacent to both $u_i$ and $u_j$, and let $g_{ij}$ be the number of common neighbors of $u_i,u_j$ outside $\{a,b\}\cup A\cup C$.
The common neighbors of $u_i$ and $u_j$ are $a$, the other two vertices of $A$, the $\lambda_{ij}$ vertices of $C$ adjacent to both, and the $g_{ij}$ common neighbors outside $\{a,b\}\cup A\cup C$.
Thus, \eqref{eq:product-local-edge-degree}, together with \eqref{eq:product-one-sided-deg}, gives
\begin{equation}
\label{eq:product-one-sided-lambda}
\lambda_{ij}+g_{ij}=1-c+r_i+r_j.
\end{equation}
Moreover, each $t\in C$ contributes to exactly one $\lambda_{ij}$, so $\displaystyle \sum_{i<j}\lambda_{ij}=c$.
Summing \eqref{eq:product-one-sided-lambda} over all six pairs and using \eqref{eq:product-one-sided-rsum}, we get $\displaystyle \sum_{i<j}g_{ij}=6-c$.
Thus, $c\le 6$.

We exclude $c=6$ and $c=5$.
If $c=6$, then every $g_{ij}=0$.
From \eqref{eq:product-one-sided-lambda}, we have $\lambda_{ij}=r_i+r_j-5$ for every pair $i<j$.
Using $\displaystyle \sum_i r_i=2c=12$, the identity $\displaystyle r_i=\sum_{j\ne i}\lambda_{ij}$ gives
\[
	r_i=\sum_{j\ne i}(r_i+r_j-5)=3r_i+(12-r_i)-15=2r_i-3.
\]
Hence $r_i=3$ for all $i$, and then $\lambda_{ij}=1$ for every pair~$i<j$.
Equation \eqref{eq:product-one-sided-deg} gives $\deg_F(u_i)=7$, so each $u_i$ has no neighbor outside $\{a\}\cup A\cup C$.

Let $t\in C$, and suppose that $t$ has a neighbor $z$ outside $\{a,b\}\cup A\cup C$.
Since $B=\emptyset$ and $z\notin C$, the outside vertex $z$ is adjacent to neither $a$ nor~$b$.
If $t$ is adjacent to $u_i,u_j\in A$, then the triples $\{b,u_i,z\}$ and $\{b,u_j,z\}$ in $N_F(t)$ force $zu_i,zu_j\in E(F)$; otherwise those triples would give triangles in~$H$.
The adjacency contradicts $g_{ij}=0$.
Hence, no vertex of $C$ has a neighbor outside $\{a,b\}\cup A\cup C$.

Now, each $t\in C$ has neighbors $a,b$, exactly two vertices of $A$, and its remaining neighbors in~$C$.
Applying \eqref{eq:product-local-edge-degree} to $bt$ shows that every $t\in C$ has exactly three neighbors in~$C$.
Also $\alpha(F[C])\le 2$, since an independent triple in $C$ would give a triangle of $H$ on three edges incident with~$b$.
Hence, $\overline{F[C]}$ is triangle-free and $2$-regular on six vertices, so $\overline{F[C]}\cong C_6$.
Thus, $F[C]$ is the triangular prism.

For any edge $tt'\in E(F[C])$, \eqref{eq:product-local-edge-degree} gives $|N_F(t)\cap N_F(t')|=4$.
The vertices $t,t'$ have no outside common neighbor, always share $a,b$, and share at most one vertex of $A$, so they must have a common neighbor in~$C$.
Thus, every edge of $F[C]$ lies in a triangle of~$F[C]$.
Such a triangle containment is impossible in the triangular prism, whose three matching edges lie in no triangle.

If $c=5$, then $\displaystyle \sum_{i<j}g_{ij}=1$.
By \eqref{eq:product-one-sided-deg}, the number of neighbors of $u_i$ outside $\{a,b\}\cup A\cup C$ is $r_i-2$, and hence the total number of incidences from $A$ to outside vertices is $\displaystyle \sum_i(r_i-2)=2c-8=2$.
An outside vertex adjacent to $s$ vertices of $A$ contributes $s$ to this incidence count and $\displaystyle \binom{s}{2}$ to $\displaystyle \sum g_{ij}$.
The total incidence count is $2$, so either one outside vertex is adjacent to two vertices of $A$, or two outside vertices are each adjacent to one vertex of~$A$.
The latter possibility contributes $0$ to $\displaystyle \sum g_{ij}$, contradicting $\displaystyle \sum g_{ij}=1$.
Hence, there is a unique outside vertex $z$ adjacent to exactly two vertices of~$A$.
Say $z$ is adjacent to $u_1,u_2$.
Then $r_1=r_2=3$ and $r_3=r_4=2$,
and $g_{12}=1$ while all other $g_{ij}$ are zero.
Substituting these values in \eqref{eq:product-one-sided-lambda} gives $\lambda_{12}=\lambda_{13}=\lambda_{14}=\lambda_{23}=\lambda_{24}=1$ and $\lambda_{34}=0$.
Let $t_{ij}$ be the unique vertex of $C$ adjacent to $u_i,u_j$, and put $S_1:=\{t_{12},t_{13},t_{14}\}$ and $S_2:=\{t_{12},t_{23},t_{24}\}$.
Let $s_i:=|N_F(z)\cap S_i|$.
No vertex outside $\{a,b\}\cup A\cup C$ is adjacent to $a$ or~$b$, and $u_1$ has no outside neighbor other than~$z$.
Therefore, the common neighbors of $u_1$ and $z$ are precisely $u_2$ together with $N_F(z)\cap S_1$.
Applying \eqref{eq:product-local-edge-degree} to $u_1z$ gives $\deg_F(z)=2s_1$.
Similarly, applying it to $u_2z$ gives $\deg_F(z)=2s_2$.
Hence, $s_1=s_2=:s$.
Since $S_1\cap S_2=\{t_{12}\}$, the vertex $z$ has at least $2s-1$ neighbors in $C$, and therefore $\deg_F(z)\ge 2+(2s-1)=2s+1$,
contradicting $\deg_F(z)=2s$.
Therefore, $c\le 4$, and so $Q=|C|\le 4$.

\item \textbf{Case 3: $(x,y)=(2,2)$.}
Let $A=\{u_1,u_2\}$ and $B=\{v_1,v_2\}$.
By \eqref{eq:product-pa-pb}, each $t\in C$ has $p_A(t)=p_B(t)$.
Let $h$ be the number of vertices $t\in C$ that are anticomplete to $A\cup B$; these are precisely the vertices with $p_A(t)=p_B(t)=2$.
Then \eqref{eq:product-q-count} becomes
\begin{equation}
\label{eq:product-balanced-q}
Q=m_F(A,B)+2h.
\end{equation}
Let $\kappa,\mu,h$ denote respectively the numbers of vertices $t\in C$ with $p_A(t)=p_B(t)=0,1,2$.
Put $r_i:=|N_F(u_i)\cap C|$.
Then
\begin{equation}
\label{eq:product-balanced-rsum}
r_1+r_2=2\kappa+\mu.
\end{equation}
The edge equation for $au_i$ gives
\begin{equation}
\label{eq:product-balanced-deg}
\deg_F(u_i)=5-|C|+2r_i.
\end{equation}
Let $q_A$ be the number of vertices of $B$ adjacent to both $u_1$ and $u_2$, and let $g_A$ be the number of common neighbors of $u_1,u_2$ outside $\{a,b\}\cup A\cup B\cup C$.
The common neighbors of $u_1,u_2$ are $a$, the $q_A$ vertices of $B$, the $\kappa$ vertices of $C$ adjacent to both $u_1,u_2$, and the $g_A$ outside vertices.
Applying \eqref{eq:product-local-edge-degree} to $u_1u_2$, using \eqref{eq:product-balanced-rsum}, \eqref{eq:product-balanced-deg}, and $|C|=\kappa+\mu+h$, gives $q_A+g_A+h=1$.
Thus, $h\le 1$.
If $h=0$, then $Q=m_F(A,B)\le 4$.
If $h=1$, then $q_A=0$, so no vertex of $B$ is adjacent to both vertices of~$A$.
It follows that $m_F(A,B)\le 2$, and \eqref{eq:product-balanced-q} gives $Q\le 2+2=4$.
\end{itemize}
\end{proof}

The corresponding count in cycle products is larger as soon as one factor is~$C_4$.

\begin{lemma}
\label{lem:cycle-product-local-c4-count}
For $m,n\ge 4$, every vertex of $C_m\square C_n$ lies on exactly $4$, $5$, or $6$ induced $4$-cycles according as neither, exactly one, or both of $m,n$ are equal to~$4$.
In particular, every vertex of $C_4\square C_n$ lies on $6$ induced $4$-cycles if $n=4$, and on $5$ induced $4$-cycles if $n>4$.
\end{lemma}

\begin{proof}
We prove the count by fixing one vertex and listing the possible induced $4$-cycles through it.
Fix a vertex $z=(i,j)$ of $C_m\square C_n$.
We use terminology motivated by the cycle decompositions of Aberle, Gold, Moshe, and Offner~\cite{AberleGMO2025}.
An \emph{elementary product square} is the $4$-cycle obtained from one edge in each cycle factor, and a \emph{$C_4$-fiber} is the $4$-cycle obtained by fixing one coordinate and traversing the other cycle factor when that factor has length~$4$.
For completeness, we verify directly the local form needed here.
An induced $4$-cycle through $z$ is determined by the two neighbors of $z$ that lie on the cycle and by their common neighbor other than~$z$.

If the two chosen neighbors of $z$ differ from $z$ in different coordinates, then they have a unique common neighbor different from~$z$.
The chosen pair gives one elementary product square.
There are two choices in the $C_m$ coordinate and two choices in the $C_n$ coordinate.
Thus, there are four elementary product squares.

If the two chosen neighbors of $z$ differ from $z$ in the same coordinate, then they have a common neighbor different from $z$ exactly when that cycle factor has length~$4$.
The same-coordinate pair contributes one $C_4$-fiber when $m=4$, and one when $n=4$.
No other pair of neighbors has a common neighbor different from~$z$, so the listed cycles are all the $4$-cycles through~$z$.
Each listed cycle is induced.
For an elementary product square, opposite vertices differ in both coordinates and hence are not adjacent in the Cartesian product.
For a $C_4$-fiber, opposite vertices are distance two in that $C_4$ factor and hence are not adjacent.
\end{proof}

\Cref{fig:c4-product-obstruction} shows the induced $4$-cycle count from \cref{lem:cycle-product-local-c4-count}, which is used below.

\begin{figure}[ht]
\centering
\begin{adjustbox}{max width=\textwidth}
\begin{tikzpicture}[
	vertex/.style={circle,draw,thick,fill=white,inner sep=1pt,minimum size=16pt,font=\scriptsize},
	zvertex/.style={circle,draw,thick,fill=black!10,inner sep=1pt,minimum size=18pt,font=\scriptsize},
	edge/.style={semithick},
	wrapedge/.style={semithick,rounded corners=3pt},
	note/.style={font=\small,align=center},
	panel/.style={font=\large\bfseries,align=center},
	cycles/.style={font=\small,align=left,inner sep=1pt}
]

\node[panel] at (-3.60,2.85) {$C_4\square C_4$};
\foreach \j/\x in {0/-5.85,1/-4.35,2/-2.85,3/-1.35} {
	\coordinate (L0\j) at (\x,1.30);
	\coordinate (L1\j) at ({\x+0.42},0.20);
	\coordinate (L2\j) at (\x,-0.90);
	\coordinate (L3\j) at ({\x-0.42},0.20);
	\draw[edge] (L0\j) -- (L1\j) -- (L2\j) -- (L3\j) -- (L0\j);
}
\foreach \i in {0,1,2,3} {
	\draw[edge] (L\i0) -- (L\i1) -- (L\i2) -- (L\i3);
}
\draw[wrapedge] (L00) -- (-5.85,1.95) -- (-1.35,1.95) -- (L03);
\draw[wrapedge] (L10) -- (-5.43,2.35) -- (-0.93,2.35) -- (L13);
\draw[wrapedge] (L20) -- (-5.85,-1.45) -- (-1.35,-1.45) -- (L23);
\draw[wrapedge] (L30) -- (-6.27,-1.80) -- (-1.77,-1.80) -- (L33);
\foreach \i/\j/\lab in {
	1/0/a,2/0/b,3/0/p,
	0/1/c,1/1/j,2/1/g,3/1/q,
	0/2/r,1/2/s,2/2/t,3/2/u,
	0/3/d,1/3/f,2/3/h,3/3/v} {
	\node[vertex] at (L\i\j) {$\lab$};
}
\node[zvertex] at (L00) {$z$};
\node[cycles] at (-3.60,-2.65) {
	\begin{tabular}{@{}l@{}}
	six induced $4$-cycles through $z$:\\
	$zajcz,\ zafdz,\ zpqcz,\ zpvdz,$\\
	$zabpz,\ zcrdz$
	\end{tabular}
};

\node[panel] at (4.10,2.85) {$C_4\square C_5$};
\foreach \j/\x in {0/1.10,1/2.60,2/4.10,3/5.60,4/7.10} {
	\coordinate (R0\j) at (\x,1.30);
	\coordinate (R1\j) at ({\x+0.42},0.20);
	\coordinate (R2\j) at (\x,-0.90);
	\coordinate (R3\j) at ({\x-0.42},0.20);
	\draw[edge] (R0\j) -- (R1\j) -- (R2\j) -- (R3\j) -- (R0\j);
}
\foreach \i in {0,1,2,3} {
	\draw[edge] (R\i0) -- (R\i1) -- (R\i2) -- (R\i3) -- (R\i4);
}
\draw[wrapedge] (R00) -- (1.10,1.95) -- (7.10,1.95) -- (R04);
\draw[wrapedge] (R10) -- (1.52,2.35) -- (7.52,2.35) -- (R14);
\draw[wrapedge] (R20) -- (1.10,-1.45) -- (7.10,-1.45) -- (R24);
\draw[wrapedge] (R30) -- (0.68,-1.80) -- (6.68,-1.80) -- (R34);
\foreach \i/\j/\lab in {
	1/0/a,2/0/b,3/0/p,
	0/1/c,1/1/j,2/1/g,3/1/q,
	0/2/r,1/2/s,2/2/t,3/2/u,
	0/3/w,1/3/x,2/3/y,3/3/m,
	0/4/d,1/4/f,2/4/h,3/4/v} {
	\node[vertex] at (R\i\j) {$\lab$};
}
\node[zvertex] at (R00) {$z$};
\node[cycles] at (4.10,-2.65) {
	\begin{tabular}{@{}l@{}}
	five induced $4$-cycles through $z$:\\
		$zajcz,\ zafdz,\ zpqcz,\ zpvdz,$\\
		$zabpz$
	\end{tabular}
};
\end{tikzpicture}
\end{adjustbox}
\caption{Induced $4$-cycle counts from \cref{lem:cycle-product-local-c4-count} for $C_4\square C_4$ and $C_4\square C_5$. A word such as $zajcz$ denotes the cycle $z-a-j-c-z$.}
\label{fig:c4-product-obstruction}
\end{figure}
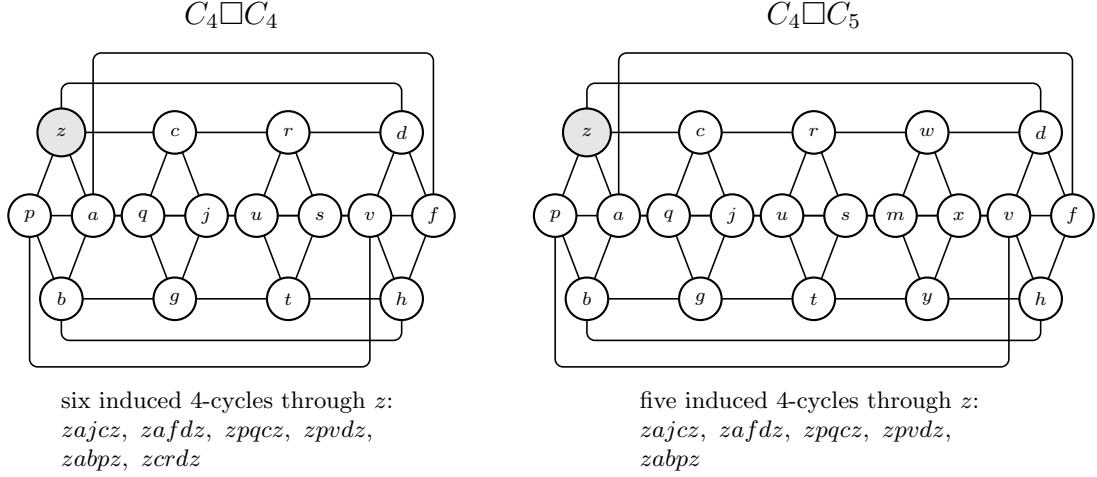

\begin{theorem}
\label{thm:c4-cycle-products-not-ts2}
For every $n\ge 4$, the graph $C_4\square C_n$ is not a $\TSsym_2$-graph.
\end{theorem}

\begin{proof}
We prove non-realizability by comparing the target's local $4$-cycle count with the complement-root bound.
Suppose to the contrary that $C_4\square C_n\cong\TSsym_2(G)$ for some graph~$G$.
Put $F:=\overline{G}$.
Then $\TSsym_2(G)=\TSsym_2(\overline{F})$.
The graph $C_4\square C_n$ is triangle-free and $4$-regular.
Indeed, both cycle factors have degree~$2$ and no triangles.
By \cref{lem:ts2-four-regular-c4-cap}, every vertex of $C_4\square C_n$ would lie on at most four induced $4$-cycles.
This upper bound contradicts \cref{lem:cycle-product-local-c4-count}, since every vertex of $C_4\square C_n$ lies on at least five induced $4$-cycles for $n\ge 4$.
\end{proof}

\section{Grid Graphs}
\label{sec:grid-realizability}

This section specializes to grids $\Gamma_{m,n}=P_m\square P_n$.
Throughout this section, we keep the standing assumption $2\le m\le n$.
A \emph{path gadget} whose maximum independent sets form a path gives the first construction.
The saturated product formula then realizes every grid $\Gamma_{m,n}$ at token value $m+n-2$.
A second two-path construction gives every grid $\Gamma_{m,n}$ at every token value $k\ge 4$.
We then prove the two-token ladder obstruction and the fixed-token classification of $\Gamma_{3,3}$.

The smallest two-dimensional grid is immediate from the cycle realization.

\begin{proposition}
\label{prop:gamma22}
$\TSsym_2(2K_2)\cong C_4\cong\Gamma_{2,2}$; in particular, $\Gamma_{2,2}$ is a $\TSsym_2$-graph.
\end{proposition}

\begin{proof}
We show the realization by using the fact that the smallest grid is a four-cycle.
We have $2K_2\cong\overline{C_4}$.
The $n=4$ case of \cref{thm:paths-cycles}\textup{(b)} gives $\TSsym_2(\overline{C_4})\cong C_4$, and $C_4\cong\Gamma_{2,2}$.
\end{proof}

\subsection{Positive Grid Constructions}

We first combine a \emph{path gadget} with the product formula.

\begin{definition}[Path gadget]
\label{def:path-gadget}
For $n\ge 2$, the graph $G_n$ has vertex set $\{x_1,\dots,x_{n-1},\,y_1,\dots,y_{n-1}\}$.
Its edge set consists of \emph{conflict pairs} $\{x_i,y_i\}$ for $1\le i\le n-1$ and \emph{crossing edges} $\{x_i,y_{i+1}\}$ for $1\le i\le n-2$.
\end{definition}

The example in \cref{fig:path-gadget} shows the gadget and its path of maximum independent sets.

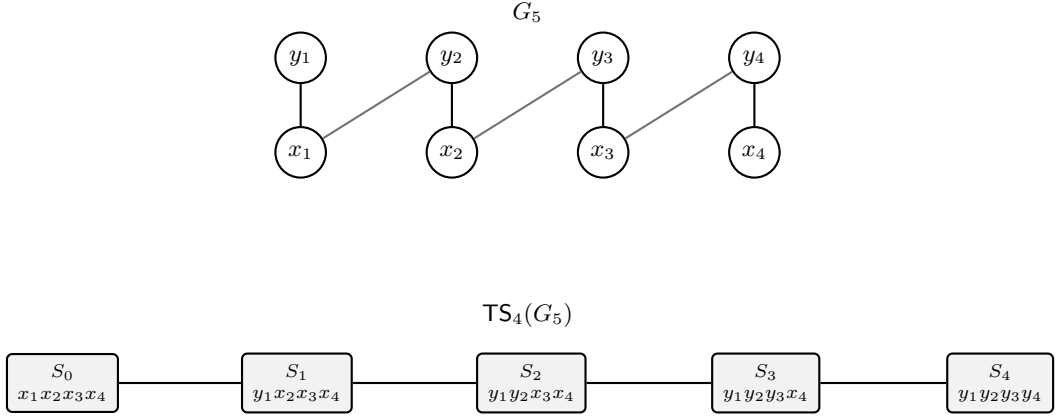
\begin{figure}[ht]
\centering
\begin{adjustbox}{max width=\textwidth}
\begin{tikzpicture}[
	vertex/.style={circle,draw,thick,fill=white,inner sep=1pt,minimum size=19pt,font=\small},
	conflict/.style={thick},
	crossing/.style={thick,draw=black!55},
	state/.style={draw,thick,rounded corners=2pt,fill=black!5,inner sep=4pt,font=\scriptsize,align=center},
	stateedge/.style={thick},
	title/.style={font=\small\bfseries}
]
\node[title] at (6.15,1.85) {$G_5$};
\foreach \i/\x in {1/3.15,2/5.15,3/7.15,4/9.15} {
	\node[vertex] (x\i) at (\x,0) {$x_{\i}$};
	\node[vertex] (y\i) at (\x,1.25) {$y_{\i}$};
	\draw[conflict] (x\i) -- (y\i);
}
\draw[crossing] (x1) -- (y2);
\draw[crossing] (x2) -- (y3);
\draw[crossing] (x3) -- (y4);

\node[title] at (6.15,-2.15) {$\TSsym_4(G_5)$};
\node[state] (s0) at (0,-3.05) {$S_0$\\$x_1x_2x_3x_4$};
\node[state] (s1) at (3.1,-3.05) {$S_1$\\$y_1x_2x_3x_4$};
\node[state] (s2) at (6.2,-3.05) {$S_2$\\$y_1y_2x_3x_4$};
\node[state] (s3) at (9.3,-3.05) {$S_3$\\$y_1y_2y_3x_4$};
\node[state] (s4) at (12.4,-3.05) {$S_4$\\$y_1y_2y_3y_4$};
\draw[stateedge] (s0) -- (s1) -- (s2) -- (s3) -- (s4);
\end{tikzpicture}
\end{adjustbox}
\caption{The path gadget for $n=5$. Vertical edges are conflict pairs, diagonal edges are crossing edges, and lower labels such as $x_1x_2x_3x_4$ denote independent sets.}
\label{fig:path-gadget}
\end{figure}

The path gadget realizes a path by forcing maximum independent sets to move one step at a time from the $x$-row to the $y$-row.

\begin{theorem}
\label{thm:path-gadget}
For every $n\ge 2$, the path gadget satisfies $\alpha(G_n)=n-1$ and $\TSsym_{n-1}(G_n)\cong P_n$.
\end{theorem}

\begin{proof}
We prove the gadget result by finding its maximum independent sets and their allowed slides.
The $n-1$ conflict pairs $\{x_i,y_i\}$ partition $V(G_n)$ into pairs, so every independent set picks at most one vertex from each pair, giving $\alpha(G_n)\le n-1$.
For $0\le j\le n-1$, define $S_j:=\{y_1,\dots,y_j,\,x_{j+1},\dots,x_{n-1}\}$; this is a set of size $n-1$.
To verify independence of~$S_j$, first consider a conflict pair $\{x_i,y_i\}$.
If both endpoints belonged to~$S_j$, then $i\ge j+1$ and $i\le j$ simultaneously.
Next consider a crossing edge $\{x_i,y_{i+1}\}$.
If both endpoints belonged to~$S_j$, then $i\ge j+1$ and $i+1\le j$.
Both alternatives are impossible, so $S_j$ is independent and $\alpha(G_n)=n-1$.

No other maximum independent set exists.
Let $S$ be a maximum independent set.
Since the conflict pairs partition $V(G_n)$ into $n-1$ pairs and $S$ contains at most one vertex from each pair, $S$ contains exactly one of $x_i,y_i$ for every~$i$.
Define $f(i)=1$ if $y_i\in S$ and $f(i)=0$ if $x_i\in S$.
If $f(i)=0$ and $f(i+1)=1$ for some~$i$, then $\{x_i,y_{i+1}\}$ is a crossing edge with both endpoints in~$S$, contradicting independence.
Hence, as $i$ increases, $f$ cannot transition from $0$ to~$1$.
The sequence therefore has the form $f=(1,\dots,1,0,\dots,0)$, and $S=S_j$ for some~$j$.

Now, $S_j$ and $S_{j+1}$ differ in exactly the element at position $j+1$: $S_j$ has $x_{j+1}$ while $S_{j+1}$ has $y_{j+1}$, connected by the conflict edge $\{x_{j+1},y_{j+1}\}$.
For $|j-j'|\ge 2$, we have $|S_j\triangle S_{j'}|=2|j-j'|\ge 4$, so they cannot be adjacent by a single token-slide.
Therefore, $\TSsym_{n-1}(G_n)$ has vertex set $\{S_0,\dots,S_{n-1}\}$ with edges $\{S_j,S_{j+1}\}$, which is $P_n$.
\end{proof}

Two path gadgets realize the grid at token value $m+n-2$.

\begin{theorem}
\label{thm:grid-realizability}
Let $2\le m\le n$.
The grid $\Gamma_{m,n}$ is a $\TSsym_{m+n-2}$-graph.
Specifically, $\TSsym_{m+n-2}(G_m\sqcup G_n)\cong \Gamma_{m,n}$ for disjoint copies of the two gadgets, where the realizing graph has $2(m-1)+2(n-1)=2(m+n-2)$ vertices.
\end{theorem}

\begin{proof}
We prove the grid realization by applying the product formula to two path gadgets, with all tokens placed in the gadgets.

By \cref{thm:path-gadget}, $\alpha(G_m)=m-1$ and $\TSsym_{m-1}(G_m)\cong P_m$, and similarly for~$G_n$.
By \cref{cor:product-saturation},
\[
\TSsym_{m+n-2}(G_m\sqcup G_n)\cong \TSsym_{m-1}(G_m)\square \TSsym_{n-1}(G_n)\cong P_m\square P_n=\Gamma_{m,n}.\qedhere
\]
\end{proof}

Padding gives all larger token values.

\begin{corollary}
\label{cor:grid-allk}
Let $2\le m\le n$.
For every $k\ge m+n-2$, there exists a graph~$G$ with $\TSsym_k(G)\cong \Gamma_{m,n}$.
\end{corollary}

\begin{proof}
We prove the padding statement by adding isolated vertices to the graph from \cref{thm:grid-realizability}.
For the graph $G_m\sqcup G_n$ from \cref{thm:grid-realizability}, we have $\alpha(G_m)=m-1$ and $\alpha(G_n)=n-1$ by \cref{thm:path-gadget}.
Thus, $\alpha(G_m\sqcup G_n)=\alpha(G_m)+\alpha(G_n)=m+n-2$.
Apply \cref{prop:padding}.
\end{proof}

A second product construction realizes every grid with four tokens.

\begin{theorem}
\label{thm:grid-four-token-product}
Let $2\le m\le n$.
We have $\TSsym_4(\overline{P_{m+1}}\sqcup\overline{P_{n+1}})\cong\Gamma_{m,n}$.
\end{theorem}

\begin{proof}
We show the four-token realization by using complement-path hosts as the two product factors.
Set $A:=\overline{P_{m+1}}$ and $B:=\overline{P_{n+1}}$.
Then $\alpha(A)=\omega(P_{m+1})=2$ and $\alpha(B)=2$ because $P_{m+1}$, $P_{n+1}$ are triangle-free paths of length $\ge 1$.
By \cref{thm:paths-cycles}, $\TSsym_2(A)\cong P_m$ and $\TSsym_2(B)\cong P_n$.
\Cref{cor:product-saturation} then gives $\TSsym_4(A\sqcup B)\cong \TSsym_2(A)\square \TSsym_2(B)\cong P_m\square P_n=\Gamma_{m,n}$.
\end{proof}

Padding the four-token construction gives all $k\ge 4$.

\begin{theorem}
\label{thm:grid-allk-four}
Let $2\le m\le n$.
For every $k\ge 4$, the grid $\Gamma_{m,n}$ is a $\TSsym_k$-graph.
\end{theorem}

\begin{proof}
We prove the all-$k$ statement by padding the four-token construction.
Set $A:=\overline{P_{m+1}}$, $B:=\overline{P_{n+1}}$, and $X:=A\sqcup B$.
As in the proof of \cref{thm:grid-four-token-product}, we have $\alpha(A)=\alpha(B)=2$, so $\alpha(X)=4$, and $\TSsym_4(X)\cong\Gamma_{m,n}$.
For $k\ge 4$, let $X'$ be obtained from $X$ by adding $k-4$ isolated vertices.
By \cref{prop:padding}, $\TSsym_k(X')\cong\TSsym_4(X)\cong\Gamma_{m,n}$.
\end{proof}

The product family also gives ladders for $k\ge 3$.

\begin{corollary}
\label{cor:ladder-kge3}
For every $n\ge 2$ and every $k\ge 3$, the ladder $\Gamma_{2,n}$ is a $\TSsym_k$-graph.
\end{corollary}

\begin{proof}
We prove the ladder case by viewing it as a path--complete graph product.
Since $P_2\cong K_2$, we have $\Gamma_{2,n}\cong P_n\square K_2$.
The result follows from \cref{cor:product-families}.
\end{proof}

\subsection{Two-Token Ladder Obstruction}

We next rule out two-token ladders by a normal-form analysis of induced $4$-cycles in a two-token graph.
The two cases are shown in \cref{fig:c4-normal-types}.
The ladder obstruction proceeds in four steps.
First, \cref{lem:c4-normal} gives the two possible normal forms for a single induced cell.
Then \cref{lem:ladder-star-phase,lem:ladder-biclique-phase,lem:ladder-mixed-interface} exclude the all-star, all-biclique, and mixed adjacent-cell patterns.
Finally, \cref{thm:no2n} assembles these exclusions.

\begin{lemma}
\label{lem:c4-normal}
Let $H$ be a graph and suppose $\TSsym_2(H)$ contains an induced $4$-cycle with vertices $I_1,I_2,I_3,I_4$ in cyclic order.
Then, up to renaming vertices of~$H$, one of the following holds:
\begin{enumerate}[label=\textup{(\alph*)}]
\item there exist distinct vertices $a,a_1,a_2,a_3,a_4$ such that $I_i=\{a,a_i\}$ for every $i\in\{1,2,3,4\}$, the edges $a_1a_2,a_2a_3,a_3a_4,a_4a_1$ belong to $E(H)$, and $aa_i\notin E(H)$ for every $i\in\{1,2,3,4\}$, while $a_1a_3,a_2a_4\notin E(H)$; or
\item there exist distinct vertices $a_1,a_2,b_1,b_2$ such that $I_1=\{a_1,b_1\}$, $I_2=\{a_1,b_2\}$, $I_3=\{a_2,b_2\}$, and $I_4=\{a_2,b_1\}$.
Moreover, $a_1a_2,b_1b_2\in E(H)$, and $a_i b_j\notin E(H)$ for all $i,j\in\{1,2\}$.
\end{enumerate}
\end{lemma}

\begin{proof}
We use two basic facts about $\TSsym_2(H)$.
First, each vertex of $\TSsym_2(H)$ is an independent $2$-set of~$H$.
Thus, if $\{p,q\}$ is one of the states, then $pq\notin E(H)$.
Second, if $\{p,q\}$ and $\{p,r\}$ are adjacent in $\TSsym_2(H)$, then the slide replaces $q$ by $r$, and so $qr\in E(H)$.

Let $I_1=\{x,y\}$.
Since $I_1I_2$ is an edge of $\TSsym_2(H)$, the two states share exactly one token.
After exchanging $x$ and $y$ if necessary, let $I_2=\{x,z\}$ with $z\notin\{x,y\}$.
The edge $I_1I_2$ slides the token from $y$ to $z$, so $yz\in E(H)$.
Hence, $\{y,z\}$ is not an independent set of~$H$.

The edge $I_2I_3$ implies that $I_2$ and $I_3$ share exactly one token.
Thus, they share either $x$ or~$z$.

\begin{itemize}
\item \textbf{Case 1: $I_2$ and $I_3$ share $x$.}
Then $I_3=\{x,w\}$ for some $w\notin\{x,y,z\}$.
Assume first that $I_3$ and $I_4$ also share~$x$.
Then $I_4=\{x,u\}$ with $u\notin\{x,y,z,w\}$.
Set $a:=x$, $a_1:=y$, $a_2:=z$, $a_3:=w$, and $a_4:=u$.
Then $I_i=\{a,a_i\}$ for every $i\in\{1,2,3,4\}$.
Since these four sets are vertices of $\TSsym_2(H)$, we have $aa_i\notin E(H)$ for every~$i$.
The four cycle edges give
\[
a_1a_2,\ a_2a_3,\ a_3a_4,\ a_4a_1\in E(H).
\]
The cycle is induced, so $I_1$ is not adjacent to $I_3$ and $I_2$ is not adjacent to~$I_4$.
Since $I_1$ and $I_3$ share $a$, this gives $a_1a_3\notin E(H)$.
Since $I_2$ and $I_4$ share $a$, this gives $a_2a_4\notin E(H)$.
This is type~\textup{(a)}.

It remains in this case to rule out the possibility that $I_3$ and $I_4$ do not share~$x$.
Then they must share~$w$, so $I_4=\{w,r\}$.
The edge $I_4I_1$ implies that $I_4$ and $I_1=\{x,y\}$ share exactly one token.
Since $w\notin\{x,y\}$, we have $r\in\{x,y\}$.
If $r=x$, then $I_4=I_3$, which is impossible.
Thus, $r=y$.
But then the edge $I_3I_4$ slides the token from $x$ to $y$, so $xy\in E(H)$.
This contradicts the fact that $I_1=\{x,y\}$ is independent.

\item \textbf{Case 2: $I_2$ and $I_3$ share $z$.}
Then $I_3=\{z,v\}$ with $v\notin\{x,z\}$.
Also $v\ne y$, because $yz\in E(H)$, so $\{y,z\}$ is not a vertex of $\TSsym_2(H)$.
Assume first that $I_3$ and $I_4$ share~$z$.
Then $I_4=\{z,q\}$.
The edge $I_4I_1$ forces $q\in\{x,y\}$, since $z\notin\{x,y\}$.
If $q=x$, then $I_4=I_2$.
If $q=y$, then $I_4=\{y,z\}$, which is not independent.
Both options are impossible.

Therefore, $I_3$ and $I_4$ share~$v$.
Write $I_4=\{v,s\}$.
The edge $I_4I_1$ forces $s\in\{x,y\}$, since $v\notin\{x,y\}$.
If $s=x$, then the edge $I_3I_4$ slides the token from $z$ to~$x$, so $xz\in E(H)$.
This contradicts the fact that $I_2=\{x,z\}$ is independent.
Therefore, $s=y$.

Thus, $I_1=\{x,y\}$, $I_2=\{x,z\}$, $I_3=\{v,z\}$, and $I_4=\{v,y\}$.
Set $a_1:=x$, $a_2:=v$, $b_1:=y$, and $b_2:=z$.
Then
\[
I_1=\{a_1,b_1\},\quad
I_2=\{a_1,b_2\},\quad
I_3=\{a_2,b_2\},\quad
I_4=\{a_2,b_1\}.
\]
The edge $I_1I_2$ gives $b_1b_2\in E(H)$.
The edge $I_2I_3$ gives $a_1a_2\in E(H)$.
Since the four displayed sets are vertices of $\TSsym_2(H)$, all four pairs $a_i b_j$ are non-edges of~$H$.
This is type~\textup{(b)}.
\end{itemize}
\end{proof}

\begin{figure}[ht]
\centering
\begin{adjustbox}{max width=\textwidth}
\begin{tikzpicture}[
	vertex/.style={circle,draw,thick,fill=white,inner sep=1pt,minimum size=18pt,font=\small},
	state/.style={draw,thick,rounded corners=2pt,fill=white,inner sep=3pt,font=\small},
	ghoststate/.style={draw=black!45,thick,rounded corners=2pt,fill=black!6,inner sep=3pt,font=\small,text=black!65},
	edge/.style={thick},
	tableline/.style={draw=black!35},
	rowlabel/.style={font=\small\bfseries},
	title/.style={font=\small\bfseries}
]
\draw[tableline] (-7.6,2.7) rectangle (7.6,-4.25);
\draw[tableline] (-7.6,2.2) -- (7.6,2.2);
\draw[tableline] (-7.6,-0.3) -- (7.6,-0.3);
\draw[tableline] (-5.65,2.7) -- (-5.65,-4.25);
\draw[tableline] (0.45,2.7) -- (0.45,-4.25);

\node[title] at (-2.6,2.45) {star type};
\node[title] at (4.0,2.45) {biclique type};
\node[rowlabel] at (-6.62,0.95) {$H$};
\node[rowlabel] at (-6.62,-2.25) {$\TSsym_2(H)$};

\node[vertex] (sta) at (-4.85,0.95) {$a$};
\node[vertex] (sta1) at (-2.6,1.75) {$a_1$};
\node[vertex] (sta2) at (-1.55,0.95) {$a_2$};
\node[vertex] (sta3) at (-2.6,0.15) {$a_3$};
\node[vertex] (sta4) at (-3.65,0.95) {$a_4$};
\draw[edge] (sta1) -- (sta2) -- (sta3) -- (sta4) -- (sta1);

\node[vertex] (ba1) at (3.05,1.55) {$a_1$};
\node[vertex] (ba2) at (5.0,1.55) {$a_2$};
\node[vertex] (bb1) at (3.05,0.35) {$b_1$};
\node[vertex] (bb2) at (5.0,0.35) {$b_2$};
\draw[edge] (ba1) -- (ba2);
\draw[edge] (bb1) -- (bb2);

\node[state] (ts1) at (-2.6,-0.8) {$\{a,a_1\}$};
\node[state] (ts2) at (-1.0,-1.75) {$\{a,a_2\}$};
\node[state] (ts3) at (-2.6,-2.7) {$\{a,a_3\}$};
\node[state] (ts4) at (-4.2,-1.75) {$\{a,a_4\}$};
\draw[edge] (ts1) -- (ts2) -- (ts3) -- (ts4) -- (ts1);
\node[ghoststate] at (-3.5,-3.65) {$\{a_1,a_3\}$};
\node[ghoststate] at (-1.7,-3.65) {$\{a_2,a_4\}$};

\node[state] (tb11) at (2.0,-0.85) {$\{a_1,b_1\}$};
\node[state] (tb12) at (6.05,-0.85) {$\{a_1,b_2\}$};
\node[state] (tb22) at (6.05,-2.65) {$\{a_2,b_2\}$};
\node[state] (tb21) at (2.0,-2.65) {$\{a_2,b_1\}$};
\draw[edge] (tb11) -- (tb12) -- (tb22) -- (tb21) -- (tb11);
\end{tikzpicture}
\end{adjustbox}
\caption{The two normal forms for an induced $4$-cycle in $\TSsym_2(H)$ from \cref{lem:c4-normal}. Gray states in the \emph{star type} are not part of the induced cycle; in the \emph{biclique type}, all cross pairs $a_i b_j$ are nonedges of~$H$.}
\label{fig:c4-normal-types}
\end{figure}

No ladder $\Gamma_{2,n}$ with $n\ge 3$ is a $\TSsym_2$-graph.
Denote the two vertices in column $i$ by $u_i$ and $v_i$.
The edge $u_i v_i$ is the $i$-th \emph{rung}.
For $1\le i<n$, the $i$-th \emph{cell} is the induced four-cycle $u_i,u_{i+1},v_{i+1},v_i,u_i$.
A cell has \emph{star type} if it is of type \textup{(a)} in \cref{lem:c4-normal}, and it has \emph{biclique type} if it is of type \textup{(b)} in \cref{lem:c4-normal}.
The notation is illustrated for $\Gamma_{2,3}$ in \cref{fig:ladder-23-notation}.
In the next three lemmas, cell types are always taken with respect to a fixed isomorphism from the token-sliding graph onto the ladder and the resulting identification of ladder vertices with independent $2$-sets of the host graph.

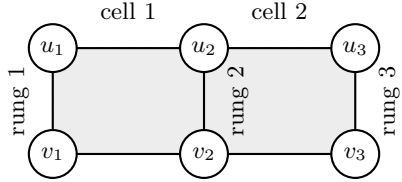
\begin{figure}[ht]
\centering
\begin{adjustbox}{max width=\textwidth}
\begin{tikzpicture}[
	vertex/.style={circle,draw,thick,fill=white,inner sep=1pt,minimum size=18pt,font=\small},
	edge/.style={thick},
	rung/.style={thick},
	cell/.style={fill=black!7,draw=black!30,dashed},
	note/.style={font=\small}
]
\coordinate (u1) at (0,1.4);
\coordinate (u2) at (2,1.4);
\coordinate (u3) at (4,1.4);
\coordinate (v1) at (0,0);
\coordinate (v2) at (2,0);
\coordinate (v3) at (4,0);
\filldraw[cell] (u1) rectangle (v2);
\filldraw[cell] (u2) rectangle (v3);
\draw[edge] (u1) -- (u2) -- (u3);
\draw[edge] (v1) -- (v2) -- (v3);
\draw[rung] (u1) -- (v1);
\draw[rung] (u2) -- (v2);
\draw[rung] (u3) -- (v3);
\node[vertex] at (u1) {$u_1$};
\node[vertex] at (u2) {$u_2$};
\node[vertex] at (u3) {$u_3$};
\node[vertex] at (v1) {$v_1$};
\node[vertex] at (v2) {$v_2$};
\node[vertex] at (v3) {$v_3$};
\node[note] at (1,1.9) {cell $1$};
\node[note] at (3,1.9) {cell $2$};
\node[note,rotate=90] at (-0.45,0.7) {rung $1$};
\node[note,rotate=90] at (2.45,0.7) {rung $2$};
\node[note,rotate=90] at (4.45,0.7) {rung $3$};
\end{tikzpicture}
\end{adjustbox}
\caption{Notation for $\Gamma_{2,3}$ in \cref{thm:no2n}: vertical edges are rungs, and shaded four-cycles are cells.}
\label{fig:ladder-23-notation}
\end{figure}

We rule out the possible cell-type patterns.
First, a ladder cannot have every cell of star type.

\begin{lemma}
\label{lem:ladder-star-phase}
For every $n\ge 2$, there is no graph~$G$ such that $\TSsym_2(G)\cong \Gamma_{2,n}$ and every cell of $\Gamma_{2,n}$ has star type.
\end{lemma}

\begin{proof}
We prove the exclusion by carrying the shared-token description along the ladder and then finding an extra token state.

Suppose to the contrary that $\TSsym_2(G)\cong \Gamma_{2,n}$ and every cell has star type.
We identify each vertex of $\Gamma_{2,n}$ with the corresponding independent $2$-set of $G$.

If $n=2$, the unique cell is star-type; its common token $q$ appears in all four vertices of $\Gamma_{2,2}$.
If $n\ge 3$, two adjacent cells of $\Gamma_{2,n}$ share a rung.
The two token sets on that rung are distinct and adjacent in $\TSsym_2(G)$, so they share exactly one token.
Since each of the two cells has one token common to all four of its vertices, the common token is the same for both cells.
The cells of $\Gamma_{2,n}$ form a path, so this common token propagates through all cells.
Thus, in every case there is a vertex $q\in V(G)$ such that every vertex of $\Gamma_{2,n}$ has the form $\{q,x\}$.
Let $Q:=\{x\in V(G):\{q,x\}\text{ is one of the token sets identified with a vertex of }\Gamma_{2,n}\}$.
Every $x\in Q$ satisfies $x\ne q$, since $\{q,x\}$ is a two-element independent set.
For distinct $x,y\in Q$, the vertices $\{q,x\}$ and $\{q,y\}$ are adjacent in $\TSsym_2(G)$ if and only if $xy\in E(G)$.
Since $\TSsym_2(G)$ is exactly $\Gamma_{2,n}$, it follows that $G[Q]\cong \Gamma_{2,n}$.

The ladder $\Gamma_{2,n}$ has two non-adjacent vertices for $n\ge 2$.
Choose distinct $x,y\in Q$ with $xy\notin E(G)$.
Then $\{x,y\}$ is an independent $2$-set of $G$, and hence a vertex of $\TSsym_2(G)$.
The vertex is not of the form $\{q,z\}$, so it is not one of the vertices identified above.
The extra token-graph vertex contradicts the assumed equality $\TSsym_2(G)\cong\Gamma_{2,n}$ and proves the lemma.
\end{proof}

The all-biclique case has the same issue: unused column labels form an extra independent set.

\begin{lemma}
\label{lem:ladder-biclique-phase}
For every $n\ge 3$, there is no graph~$G$ such that $\TSsym_2(G)\cong \Gamma_{2,n}$ and every cell of $\Gamma_{2,n}$ has biclique type.
\end{lemma}

\begin{proof}
We prove the biclique exclusion by carrying the forced labels across the ladder and then finding an unused cross choice.

Suppose to the contrary that $\TSsym_2(G)\cong \Gamma_{2,n}$ and every cell has biclique type.
We identify each vertex of $\Gamma_{2,n}$ with the corresponding independent $2$-set of $G$.

Choose the first cell.
By \cref{lem:c4-normal} type~\textup{(b)}, after possibly swapping the two rows and relabeling, assume $u_1=\{r_0,c_1\}$, $v_1=\{r_1,c_1\}$, $v_2=\{r_1,c_2\}$, and $u_2=\{r_0,c_2\}$.
We prove by induction that, for every $i\in\{1,\ldots,n\}$,
$u_i=\{r_0,c_i\}$ and $v_i=\{r_1,c_i\}$
for some label $c_i$.
The statement holds for $i=1,2$.
Suppose it holds for some $i$ with $2\le i<n$.
The cell between columns $i$ and $i+1$ has biclique type and contains the rung $\{r_0,c_i\},\{r_1,c_i\}$.
After swapping the two coordinate classes in the biclique normal form if necessary, the fixed old rung is the common-second-coordinate edge; hence the opposite rung keeps the first coordinates $r_0,r_1$ and replaces the common second coordinate $c_i$ by one new label.
Hence, the two new vertices in column $i+1$ have the form $\{r_0,c_{i+1}\},\{r_1,c_{i+1}\}$
for some label $c_{i+1}$.
The label $c_{i+1}$ is distinct from $r_0$, $r_1$, and $c_i$.
Otherwise, one of the displayed two-element independent sets is not a two-element set, or the new column repeats the previous column.
The top-row edge from $u_i$ to $u_{i+1}$ forces $u_{i+1}=\{r_0,c_{i+1}\}$, because $u_i=\{r_0,c_i\}$ is disjoint from $\{r_1,c_{i+1}\}$.
Thus, $v_{i+1}=\{r_1,c_{i+1}\}$.
The induction follows.
The labels $c_1,\ldots,c_n$ are pairwise distinct, because the token sets $u_1,\ldots,u_n$ are pairwise distinct.
Moreover, none of the labels $c_i$ equals $r_0$ or $r_1$, because the displayed objects are independent $2$-sets.

Since $n\ge 3$, the vertices $u_1=\{r_0,c_1\}$ and $u_3=\{r_0,c_3\}$ are non-adjacent in $\Gamma_{2,n}$.
If $c_1c_3\in E(G)$, then $u_1$ and $u_3$ would be adjacent in $\TSsym_2(G)$.
Therefore, $c_1c_3\notin E(G)$, and $\{c_1,c_3\}$ is an independent $2$-set of $G$.
The independent set is not one of the displayed sets $\{r_0,c_i\}$ or $\{r_1,c_i\}$.
The extra token-graph vertex contradicts the assumed equality $\TSsym_2(G)\cong\Gamma_{2,n}$ and proves the lemma.
\end{proof}

Adjacent cells also cannot switch type: the shared rung would force extra token-graph vertices with adjacencies not present in the ladder.

\begin{lemma}
\label{lem:ladder-mixed-interface}
No graph~$G$ satisfies $\TSsym_2(G)\cong \Gamma_{2,n}$ for some $n\ge 3$ with two adjacent cells of different types.
\end{lemma}

\begin{proof}
We prove the mixed-interface exclusion by showing that a type change forces extra token states.
Suppose to the contrary that such a graph $G$ exists.
By reversing the order of the columns if necessary, assume that the left cell has star type and the right cell has biclique type.
We identify each ladder vertex with the corresponding independent $2$-set of $G$.

Let the left cell have cyclic order $A,B,C,D,A$, where the edge $BC$ is the shared rung with the right cell.
Since the left cell has star type, assume
$A=\{q,a\}$, $B=\{q,b\}$, $C=\{q,c\}$, and $D=\{q,d\}$.
The right cell has cyclic order $B,R,S,C,B$.
Since the right cell has biclique type, its two new vertices have the form $R=\{\rho,b\}$ and $S=\{\rho,c\}$.
The label $\rho$ is distinct from $q,a,b,c,d$.
Indeed, $\rho=q$ repeats $B$ or $C$, while $\rho=b$ or $\rho=c$ gives a one-element token set.
If $\rho=a$, then $R=\{a,b\}$, but the edge $AB$ of the left cell forces $ab\in E(G)$, so $R$ is not independent.
If $\rho=d$, then $S=\{d,c\}$, but the edge $CD$ of the left cell forces $cd\in E(G)$, so $S$ is not independent.

The four edges of the left cell give $ab,bc,cd,da\in E(G)$.
The two missing diagonals of the induced left cell give
$ac,bd\notin E(G)$.
Therefore, $Z_1:=\{a,c\}$ and $Z_2:=\{b,d\}$ are vertices of $\TSsym_2(G)$.
The edge $BR$ of the right cell gives $q\rho\in E(G)$.

We show that $a\rho\in E(G)$.
Suppose that $a\rho\notin E(G)$.
Then $U:=\{a,\rho\}$ is a vertex of $\TSsym_2(G)$.
The vertex $U$ is adjacent to $A=\{q,a\}$, because $q\rho\in E(G)$.
It is also adjacent to $R=\{\rho,b\}$, because $ab\in E(G)$.
Thus, $U$ is a common neighbor of $A$ and $R$.
In the ladder, $A$ and $R$ lie in the same row at distance two, and their unique common neighbor is $B$.
Since $a\ne q,b$ and $\rho\ne q,b$, the set $U$ is not equal to $B=\{q,b\}$, contradicting the assumed equality with the ladder.
Hence, $a\rho\in E(G)$.
By applying the same argument to $D$ and $S$, we get $d\rho\in E(G)$.

Now, $Z_1=\{a,c\}$ is adjacent to $S=\{\rho,c\}$, because $a\rho\in E(G)$.
Similarly, $Z_2=\{b,d\}$ is adjacent to $R=\{\rho,b\}$, because $d\rho\in E(G)$.
Moreover, $Z_1\ne C$ because $a\ne q$, and $Z_1\ne R$ because $\rho$ is distinct from $a,c$ and $b\ne a,c$.
Similarly, $Z_2\ne B$ and $Z_2\ne S$.
The remaining possible equalities with $B,C,R,S$ are excluded by the distinctness of $q,a,b,c,d,\rho$, and $Z_1=Z_2$ would force one of $a=b,a=d,c=b,c=d$.
Thus, $Z_1$ and $Z_2$ are distinct from each other and from the relevant ladder vertices $B,C,R,S$.

\begin{itemize}
\item \textbf{Case 1.} The right cell is the last cell of the ladder.
Then $R$ and $S$ are boundary vertices of degree two in $\Gamma_{2,n}$.
The vertex $S$ is already adjacent to $C$ and $R$, and $Z_1$ is distinct from both.
Thus, $S$ has a third neighbor in $\TSsym_2(G)$, although the corresponding boundary ladder vertex has degree two.
This contradicts the assumed equality with the ladder.

\item \textbf{Case 2.} The right cell is not the last cell of the ladder.
Then $R$ and $S$ are internal vertices of degree three in $\Gamma_{2,n}$.
The vertex $S$ is already adjacent to $C$ and $R$, so its unique neighbor outside the right cell must be $Z_1$.
Similarly, the unique neighbor of $R$ outside the right cell must be $Z_2$.
These two forced outside neighbors must be the two vertices of the next rung of the ladder.
Since a rung is an edge, $Z_1$ and $Z_2$ must be adjacent in $\TSsym_2(G)$.
But $Z_1=\{a,c\}$ and $Z_2=\{b,d\}$ are disjoint token sets.
Two vertices of $\TSsym_2(G)$ can be adjacent only if their token sets share one token, a contradiction.
\end{itemize}

Both cases are impossible.
\end{proof}

These three lemmas cover all possible induced $4$-cycle normal forms for a ladder.

\begin{theorem}
\label{thm:no2n}
For every $n\ge 3$, the ladder $\Gamma_{2,n}$ is not a $\TSsym_2$-graph.
\end{theorem}

\begin{proof}
We prove the theorem by reducing every ladder realization to one of the three cell-type patterns already ruled out.
Suppose to the contrary that $\TSsym_2(G)\cong \Gamma_{2,n}$ for some graph~$G$.
Since $n\ge 3$, the ladder $\Gamma_{2,n}$ has at least two cells.
By \cref{lem:c4-normal}, every cell has either star type or biclique type.
By \cref{lem:ladder-mixed-interface}, two adjacent cells cannot have different types.
Since the cells of $\Gamma_{2,n}$ form a path, all cells have the same type.
\begin{itemize}
\item If all cells have star type, then \cref{lem:ladder-star-phase} gives a contradiction.

\item If all cells have biclique type, then \cref{lem:ladder-biclique-phase} gives a contradiction.
\end{itemize}
The two cases exhaust all possibilities, so no such graph $G$ exists.
\end{proof}

The preceding results give the fixed-token ladder classification.

\begin{theorem}
\label{thm:ladder-bounds}
For $n\ge 2$ and $k\ge 2$, the ladder $\Gamma_{2,n}$ is a $\TSsym_k$-graph if and only if either $k\ge 3$ or $(n,k)=(2,2)$.
\end{theorem}

\begin{proof}
We prove the classification by combining the small case, the two-token obstruction, and the positive construction for $k\ge 3$.
If $k=2$, then \cref{prop:gamma22} realizes exactly the case $n=2$, while \cref{thm:no2n} rules out all $n\ge 3$.
If $k\ge 3$, then \cref{cor:ladder-kge3} realizes every ladder $\Gamma_{2,n}$ with $n\ge 2$.
\end{proof}

\begin{remark}
\label{rem:smallest-non-ts2-grid}
Under the standing assumption $2\le m\le n$, the grid $\Gamma_{2,3}$ is the smallest grid by number of vertices that is not a $\TSsym_2$-graph.
Indeed, the only smaller grid permitted by the assumption is $\Gamma_{2,2}\cong C_4$, which is realizable by \cref{prop:gamma22}.
\end{remark}

\subsection{The 3-by-3 Grid}
\label{subsec:gamma33}

We now settle the first non-ladder width-three grid.
The two-token proof uses a complement root and the degree formula from \cref{prop:per-edge-degree}.
The three-token proof uses the $3$-Gallai graph.

To rule out $\Gamma_{3,3}$ at two tokens, we need one more identity that works even before reducing the host.
It counts the boundary of a \emph{root star} in the target and converts large root degree into too many internal target edges.

\begin{lemma}
\label{lem:star-set-identity}
Suppose $\phi:\TSsym_2(\overline{F})\xrightarrow{\cong} \Gamma$ is an isomorphism, and identify $V(\TSsym_2(\overline{F}))$ with $E(F)$.
Let $v\in V(F)$ have degree $d$, and put $\widehat S_v:=\{vx:x\in N_F(v)\}$ and $S_v:=\phi(\widehat S_v)\subseteq V(\Gamma)$.
Let $W:=V(F)\setminus N_F[v]$ and let $b:=|E(F[W])|$.
Then
\[
\partial_\Gamma(S_v)+b=|E(F)|-d-\binom d2+|E(\Gamma[S_v])|,
\]
where $\partial_\Gamma(S_v)$ denotes the number of $\Gamma$-edges with exactly one endpoint in~$S_v$.
\end{lemma}

\begin{proof}
We prove the identity by separating the target edges inside the root star from those leaving it.
Two star edges $vx,vy\in \widehat S_v$ are adjacent in $\TSsym_2(\overline{F})$ exactly when $xy\notin E(F)$.
Thus, $\displaystyle |E(\Gamma[S_v])|=\binom d2-|E(F[N_F(v)])|$.
Moreover, under~$\phi$, the $\Gamma$-edges leaving $S_v$ correspond exactly to the $F$-edges with one endpoint in $N_F(v)$ and the other in~$W$.
Indeed, if $xy\in E(F)$ with $x\in N_F(v)$ and $y\in W$, then the two root edges $vx$ and $xy$ share~$x$, and the other endpoints $v,y$ are non-adjacent in~$F$, so they are adjacent in $\TSsym_2(\overline{F})$.
Conversely, any $\Gamma$-edge leaving $S_v$ has preimage a $\TSsym_2(\overline{F})$-edge from some $vx\in\widehat S_v$ to a root edge outside $\widehat S_v$.
That root edge must share~$x$, say it is $xy$; if it shared~$v$, it would also lie in $\widehat S_v$.
Token-sliding adjacency forces $vy\notin E(F)$, and hence $y\in W$.
Let $c$ be the number of $F$-edges with one endpoint in $N_F(v)$ and the other in~$W$.
Therefore, $\partial_\Gamma(S_v)=c$.
Since $|E(F)|=d+|E(F[N_F(v)])|+c+b$, substitution gives the identity.
\end{proof}

The first width-three exclusion gives the main two-token obstruction for $\Gamma_{3,3}$.
If a complement root has nine edges and maximum degree at most~$4$, then the unique center of $\Gamma_{3,3}$ forces an impossible parity pattern around its preimage edge.

\begin{lemma}
\label{lem:no33-center-obstruction}
Let $F$ be a graph with $|E(F)|=9$ and $\Delta(F)\le 4$.
Then $\TSsym_2(\overline{F})\not\cong \Gamma_{3,3}$.
\end{lemma}

\begin{proof}
We prove the claim by mapping the grid center to a root edge and then counting the four odd-degree side vertices around it.
Suppose to the contrary that $\TSsym_2(\overline{F})\cong \Gamma_{3,3}$.
Let $\Gamma:=\Gamma_{3,3}$, and let $c$ be the unique degree-$4$ vertex of~$\Gamma$.
Fix an isomorphism $\phi:\TSsym_2(\overline{F})\xrightarrow{\cong}\Gamma$.
The vertices of $\TSsym_2(\overline{F})$ are exactly the edges of~$F$, so we regard $\phi$ as a bijection from $E(F)$ to $V(\Gamma)$.

Let $ab\in E(F)$ be the edge such that $\phi(ab)=c$.
Call an edge $xy\in E(F)$ \emph{mixed} if $\deg_F(x)$ and $\deg_F(y)$ have opposite parity.
By \cref{prop:per-edge-degree}, every edge $xy\in E(F)$ satisfies $\deg_\Gamma(\phi(xy))=\deg_F(x)+\deg_F(y)-2-2|N_F(x)\cap N_F(y)|$.
Thus, $xy$ is mixed if and only if $\phi(xy)$ has odd degree in~$\Gamma$.

In $\Gamma_{3,3}$, the odd-degree vertices are precisely the four neighbors of the center~$c$, and these four vertices are pairwise non-adjacent.
Hence, the mixed edges of~$F$ are precisely the four edges mapped by $\phi$ to $N_\Gamma(c)$.
In particular, every mixed edge is adjacent to $ab$ in $\TSsym_2(\overline{F})$.
By \cref{lem:line-avis}, any root edge adjacent to $ab$ in the token-sliding graph must share an endpoint with~$ab$.
Therefore, every mixed edge is incident with $a$ or with~$b$.

Put $d_a:=\deg_F(a)$, $d_b:=\deg_F(b)$, and $t_0:=t_F(ab)$.
Since $\phi(ab)=c$ and $\deg_\Gamma(c)=4$, \cref{prop:per-edge-degree} gives $4=d_a+d_b-2-2t_0$.
Moreover, $ab$ is not mixed, because $c$ has even degree.
Thus, $d_a$ and $d_b$ have the same parity.
The displayed equation gives $d_a+d_b\ge 6$, and $d_a,d_b\le 4$ by hypothesis.
Therefore, the unordered pair $\{d_a,d_b\}$ is one of $\{3,3\}$, $\{2,4\}$, and $\{4,4\}$.

Let $r$ be the number of mixed edges incident with~$a$, and let $s:=4-r$.
Let the mixed edges be $ax_1,\ldots,ax_r$ and $by_1,\ldots,by_s$.
The vertices $x_1,\ldots,x_r$ are pairwise distinct and none of them is~$b$.
The vertices $y_1,\ldots,y_s$ are pairwise distinct and none of them is~$a$.
A vertex $x_i$ may equal some vertex $y_j$.

A local estimate controls the degrees of the endpoints away from $a$ and~$b$.
Fix $i\in\{1,\ldots,r\}$.
For every $h\ne i$, the two edges $ax_i$ and $ax_h$ correspond to two distinct degree-$3$ vertices of~$\Gamma$.
These two grid vertices are non-adjacent.
Thus, $ax_i$ and $ax_h$ are not adjacent in $\TSsym_2(\overline{F})$.
Since the two root edges share the endpoint~$a$, they would be adjacent in $\TSsym_2(\overline{F})$ exactly when $x_i x_h\notin E(F)$.
Hence $x_i x_h\in E(F)$, and $|N_F(a)\cap N_F(x_i)|\ge r-1$.
Applying \cref{prop:per-edge-degree} to the mixed edge $ax_i$, whose image has degree~$3$ in~$\Gamma$, gives
$3=d_a+\deg_F(x_i)-2-2|N_F(a)\cap N_F(x_i)|$.
Hence,
\begin{equation}
\deg_F(x_i)\ge 5-d_a+2(r-1).
\label{eq:no33-x-est}
\end{equation}
The symmetric argument gives, for every $j\in\{1,\ldots,s\}$,
\begin{equation}
\deg_F(y_j)\ge 5-d_b+2(s-1).
\label{eq:no33-y-est}
\end{equation}

\begin{itemize}
\item \textbf{Case 1: $\{d_a,d_b\}=\{2,4\}$.}
By symmetry, assume $d_a=2$ and $d_b=4$.
Since $ab$ is incident with both $a$ and $b$, we have $r\le d_a-1=1$ and $s\le d_b-1=3$.
Together with $r+s=4$, this gives $r=1$ and $s=3$.
For every $j\in\{1,2,3\}$, inequality~\eqref{eq:no33-y-est} gives $\deg_F(y_j)\ge 5-4+2(3-1)=5$,
contradicting $\Delta(F)\le 4$.

\item \textbf{Case 2: $d_a=d_b=3$.}
Since $ab$ is incident with both $a$ and $b$, we have $r\le 2$ and $s\le 2$.
Together with $r+s=4$, this gives $r=s=2$.
By inequalities~\eqref{eq:no33-x-est} and~\eqref{eq:no33-y-est}, each of $x_1,x_2,y_1,y_2$ has degree at least~$4$.

We claim that $x_1,x_2,y_1,y_2$ are four distinct vertices.
The vertices $x_1,x_2$ are distinct, and the vertices $y_1,y_2$ are distinct.
If $x_i=y_j$ for some $i,j$, then $x_i$ is a common neighbor of $a$ and~$b$.
Indeed, $ax_i\in E(F)$ and $bx_i=by_j\in E(F)$.
However, the equation $4=d_a+d_b-2-2t_0$ with $d_a=d_b=3$ gives $t_0=0$.
The contradiction proves the claim.

Therefore, $\displaystyle \sum_{v\in V(F)}\deg_F(v)\ge d_a+d_b+\deg_F(x_1)+\deg_F(x_2)+\deg_F(y_1)+\deg_F(y_2)\ge 3+3+4\cdot 4=22$.
The last inequality is impossible because $\displaystyle \sum_{v\in V(F)}\deg_F(v)=2|E(F)|=18$.

\item \textbf{Case 3: $d_a=d_b=4$.}
Since $ab$ is incident with both $a$ and $b$, we have $r,s\le 3$.
Thus, $r\in\{1,2,3\}$.
If $r=1$, then $s=3$, and inequality~\eqref{eq:no33-y-est} gives $\deg_F(y_j)\ge 5-4+2(3-1)=5$ for every $j\in\{1,2,3\}$, contradicting $\Delta(F)\le 4$.
The case $r=3$ is symmetric.
Thus, $r=s=2$.

By inequalities~\eqref{eq:no33-x-est} and~\eqref{eq:no33-y-est}, each of $x_1,x_2,y_1,y_2$ has degree at least~$3$.
We claim that $x_i\ne y_j$ for all $i,j\in\{1,2\}$.
Suppose to the contrary that $x_i=y_j$.
Let $i'\in\{1,2\}\setminus\{i\}$.
As above, the two mixed edges $ax_i$ and $ax_{i'}$ are not adjacent in $\TSsym_2(\overline{F})$, so $x_i x_{i'}\in E(F)$.
Since $x_i=y_j$, we also have $bx_i\in E(F)$.
Hence, $b$ and $x_{i'}$ are two distinct common neighbors of $a$ and~$x_i$.
Therefore, $|N_F(a)\cap N_F(x_i)|\ge 2$.
Applying \cref{prop:per-edge-degree} to $ax_i$ gives $\deg_F(x_i)=5-d_a+2|N_F(a)\cap N_F(x_i)|\ge 5-4+2\cdot 2=5$,
contradicting $\Delta(F)\le 4$.
Thus, $x_1,x_2,y_1,y_2$ are four distinct vertices.

It follows that $\displaystyle \sum_{v\in V(F)}\deg_F(v)\ge d_a+d_b+\deg_F(x_1)+\deg_F(x_2)+\deg_F(y_1)+\deg_F(y_2)\ge 4+4+4\cdot 3=20$.
The last inequality is impossible because $\displaystyle \sum_{v\in V(F)}\deg_F(v)=2|E(F)|=18$.
\end{itemize}

All possible cases lead to contradictions.
Therefore, $\TSsym_2(\overline{F})\not\cong \Gamma_{3,3}$.
\end{proof}

The previous obstruction assumes a maximum-degree bound on the complement root.
The following elementary fact about $\Gamma_{3,3}$ gives the target-side estimate needed to prove that bound for any unreduced root.

\begin{lemma}
\label{lem:gamma33-five-vertex-bound}
In $\Gamma_{3,3}$, any two distinct vertices have at most two common neighbors, and every five-vertex set spans at most five edges.
\end{lemma}

\begin{proof}
We prove the two target-side bounds directly from the coordinate structure of $\Gamma_{3,3}$.
Represent vertices as pairs $(i,j)$ with $i,j\in\{0,1,2\}$.
A common neighbor of two vertices must be obtained from each by changing exactly one coordinate by~$1$.
Thus, two distinct vertices have common neighbors only when they are at graph distance~$2$, and then they have either one common neighbor in the same row or column, or two common neighbors when they are diagonally opposite corners of a grid square.
Hence, no two distinct vertices have more than two common neighbors.

Let $S\subseteq V(\Gamma_{3,3})$ with $|S|=5$.
The induced graph $\Gamma_{3,3}[S]$ is bipartite.
If it had at least six edges, its bipartition inside~$S$ would have sizes $2$ and~$3$ and all possible cross-edges would have to be present, so $\Gamma_{3,3}[S]\cong K_{2,3}$.
The two vertices in the part of size~$2$ would then have the three vertices in the other part as common neighbors, contradicting the common-neighbor bound above.
Therefore, $|E(\Gamma_{3,3}[S])|\le 5$.
\end{proof}

We can now finish the two-token exclusion by combining the target-side bounds of \cref{lem:gamma33-five-vertex-bound} with the center obstruction of \cref{lem:no33-center-obstruction}.

\begin{theorem}
\label{thm:no33}
The grid $\Gamma_{3,3}$ is not a $\TSsym_2$-graph.
\end{theorem}

\begin{proof}
We prove the exclusion by first bounding the maximum degree of an unreduced complement root and then applying the center obstruction.
Suppose to the contrary that $\TSsym_2(H)\cong\Gamma_{3,3}$ for some graph~$H$.
Let $F_0:=\overline{H}$ and let $\Gamma:=\Gamma_{3,3}$.
In this proof we intentionally use the unreduced complement rather than the reduced complement-root convention above; isolated vertices of $F_0$ have no corresponding token states and do not affect the edge-count and maximum-degree arguments below.
Then $\TSsym_2(H)=\TSsym_2(\overline{F_0})$.
The vertices of $\TSsym_2(\overline{F_0})$ are exactly the edges of~$F_0$.
Thus, $|E(F_0)|=|V(\Gamma)|=9$.

We show that $\Delta(F_0)\le 4$.
Fix an isomorphism $\phi:\TSsym_2(\overline{F_0})\xrightarrow{\cong}\Gamma$.
Let $v\in V(F_0)$ have degree~$d$.
Apply \cref{lem:star-set-identity} to $v$ and~$\phi$.
The left-hand side in \cref{lem:star-set-identity} is non-negative.
If $d=5$, then \cref{lem:gamma33-five-vertex-bound} gives $|E(\Gamma[S_v])|\le 5$.
Therefore, $\displaystyle \partial_\Gamma(S_v)+b\le 9-5-\binom52+5=-1$,
impossible.
If $d\ge 6$, then $\Gamma[S_v]$ is an induced subgraph of the bipartite planar grid $\Gamma_{3,3}$.
Hence $\Gamma[S_v]$ is a simple bipartite planar graph on $d\ge 3$ vertices.
The bound $|E(\Gamma[S_v])|\le 2d-4$ also holds when this induced subgraph is disconnected, by applying the bipartite planar bound componentwise and using the trivial bound on components with at most two vertices.
Thus, $\displaystyle \partial_\Gamma(S_v)+b\le 9-d-\binom d2+(2d-4)=5+d-\binom d2<0$,
again impossible.
Hence, $\Delta(F_0)\le 4$.

Together with $|E(F_0)|=9$, the bound $\Delta(F_0)\le 4$ contradicts \cref{lem:no33-center-obstruction}.
\end{proof}

For a graph $F$, the \emph{$3$-Gallai graph} $\Theta(F)$ has the triangles of~$F$ as vertices.
Two triangles $T_1,T_2$ are adjacent in $\Theta(F)$ if $|T_1\cap T_2|=2$ and $F[T_1\cup T_2]\cong K_4-e$.
Equivalently, if $T_1=\{a,b,x\}$ and $T_2=\{a,b,y\}$ share the edge $ab$, then $T_1T_2\in E(\Theta(F))$ exactly when $xy\notin E(F)$.
The graph $\Theta(F)$ is Prisner's $k=3$ case of the $k$-Gallai graph~\cite[Def.~15.1, Sec.~15.1]{Prisner1995GraphDynamics}.
Independent triples in a graph $G$ are triangles in $\overline G$.
Thus $\TSsym_3(G)$ is the $3$-Gallai graph $\Theta(\overline G)$.

\begin{theorem}
\label{thm:no33-ts3}
The grid $\Gamma_{3,3}$ is not a $\TSsym_3$-graph.
\end{theorem}

\begin{proof}
Suppose to the contrary that $\Gamma_{3,3}\cong\TSsym_3(G)$ for some graph~$G$.
Put $F:=\overline G$ and identify $\TSsym_3(G)$ with $\Theta(F)$ as above.
Fix an isomorphism $\Theta(F)\xrightarrow{\cong}\Gamma_{3,3}$.
Let $M=\{a,b,c\}$ be the triangle of~$F$ corresponding to the unique degree-$4$ vertex of $\Gamma_{3,3}$, namely its center.

The four neighbors of $M$ in $\Theta(F)$ correspond to the four side vertices of the grid.
Each such neighbor shares one of the three edges of~$M$.
Since there are four such neighbors but only three edges of~$M$, two of them share the same edge of~$M$.
After relabeling so that the repeated edge is $ab$, let these two neighbors be
$P=\{a,b,x\}$ and $Q=\{a,b,y\}$,
where $x,y\notin\{a,b,c\}$ and $x\ne y$.

The four side vertices of $\Gamma_{3,3}$ are pairwise non-adjacent, so $P$ and $Q$ are non-adjacent in $\Theta(F)$.
Since $P$ and $Q$ share the edge $ab$, this non-adjacency forces $xy\in E(F)$.
Therefore $F[\{a,b,x,y\}]\cong K_4$, so $U=\{a,x,y\}$ and $V=\{b,x,y\}$ are also triangles of~$F$ and hence vertices of $\Theta(F)$.
They are distinct because $a\ne b$.
Moreover, each shares exactly one vertex with~$M$, so neither is $M$ nor a neighbor of~$M$ in $\Theta(F)$.
Also $F[\{a,b,x,y\}]\cong K_4$, so neither $U$ nor $V$ is adjacent in $\Theta(F)$ to either $P$ or~$Q$.

Under the fixed isomorphism with $\Gamma_{3,3}$, the vertices other than the center and its four side vertices are exactly the four corners.
Thus, the images of $U$ and $V$ must be distinct corners, and each must be non-adjacent to the two side vertices corresponding to $P$ and~$Q$.

The required pair of corners does not exist in $\Gamma_{3,3}$.
Each corner is adjacent exactly to the two side vertices incident with it.
Hence, for two distinct side vertices, the number of corners non-adjacent to both is $0$ if the side vertices are opposite across the center, and $1$ if they are consecutive around the center.
In particular, at most one corner is non-adjacent to both selected side vertices.

The distinct vertices $U$ and $V$ would therefore require two distinct such corners, a contradiction.
Hence, $\Gamma_{3,3}$ is not a $\TSsym_3$-graph.
\end{proof}

Combining the two-token and three-token obstructions with the positive construction in \cref{thm:grid-allk-four} gives the fixed-token classification of $\Gamma_{3,3}$.

\begin{theorem}
\label{thm:gamma33-fixed-k}
For $k\ge 2$, the grid $\Gamma_{3,3}$ is a $\TSsym_k$-graph if and only if $k\ge 4$.
\end{theorem}

\begin{proof}
We prove the two implications separately.
\begin{itemize}
\item[$(\Rightarrow)$]
By \cref{thm:no33}, $\Gamma_{3,3}$ is not a $\TSsym_2$-graph.
By \cref{thm:no33-ts3}, $\Gamma_{3,3}$ is not a $\TSsym_3$-graph.
Thus, any realization with $k\ge 2$ must have $k\ge 4$.

\item[$(\Leftarrow)$]
For every $k\ge 4$, \cref{thm:grid-allk-four} gives that $\Gamma_{3,3}$ is a $\TSsym_k$-graph.
\end{itemize}
\end{proof}

\section{Concluding Remarks}
\label{sec:conclusion}

This paper treats token-sliding realizability as an inverse structural problem for reconfiguration graphs.
The results identify both robust construction principles and sharp local obstructions.
Complement families give exact positive and negative targets, the disjoint-union product formula builds Cartesian-product state spaces, and grids test where product behavior persists or fails at small token values.
For grids under the standing convention $2\le m\le n$, we realize every $\Gamma_{m,n}$ at $k=m+n-2$ and for all $k\ge 4$, classify ladders by fixed token number, and settle $\Gamma_{3,3}$ at all fixed token values $k\ge 2$ (\cref{tab:realizability-summary}).

The remaining questions are small-token product cases, not exceptions to the large-token grid realization.
They are listed in \cref{tab:remaining-product-cases}.
\begin{table}[ht]
\centering
\renewcommand{\arraystretch}{1.12}
\begin{tabularx}{\textwidth}{@{}L{0.25\textwidth}L{0.18\textwidth}Y@{}}
\toprule
Family & Token value & Remaining cases \\
\midrule
$C_m\square P_n$ & $k=2$ & $C_3\square P_n$ with $n\ge 6$; $C_m\square P_n$ with $m\ge 4$ and $n\ge 3$ \\
\cmidrule(lr){1-3}
$C_m\square P_n$ & $k=3$ & $C_m\square P_n$ with $m\ge 4$ and $n\ge 3$ \\
\cmidrule(lr){1-3}
$C_m\square C_n$ & $k=2$ & $C_3\square C_n$ with $n\ge 4$; $C_m\square C_n$ with $m,n\ge 5$, up to symmetry \\
\cmidrule(lr){1-3}
$C_m\square C_n$ & $k=3$ & $C_m\square C_n$ with $m,n\ge 4$, up to symmetry \\
\cmidrule(lr){1-3}
$\Gamma_{m,n}=P_m\square P_n$ & $k=2,3$ & $3\le m\le n$, except $(m,n)=(3,3)$ \\
\bottomrule
\end{tabularx}
\caption{Remaining small-token Cartesian-product cases. Cycle--path rows summarize \cref{rem:cycle-path-frontier}; the grid row excludes the ladder case from \cref{thm:ladder-bounds} and $\Gamma_{3,3}$ from \cref{thm:gamma33-fixed-k}.}
\label{tab:remaining-product-cases}
\end{table}
The next width-three cases are $\Gamma_{3,n}$ with $n\ge 4$ at $k=2$ and $k=3$.
These questions require complement-root tools beyond the product, ladder, and $\Gamma_{3,3}$ arguments developed here.

\section*{Acknowledgements}

The work began at the Workshop on Graphs and Geometric Algorithms (WOGGA 3), OIST, Japan (2024).
We thank Meike Hatzel and Marcelo Garlet Milani for their fruitful discussions during the workshop.
The research was supported by the Vietnam National University, Hanoi under the project QG.25.07 ``A study on reconfiguration problems from algorithmic and graph-theoretic perspectives''.

\printbibliography

@article{AvisHoang2024,
	title          = {A Note on Acyclic Token Sliding Reconfiguration Graphs of Independent Sets},
	author         = {David Avis and Duc A. Hoang},
	year           = {2024},
	journal        = {Ars Combinatoria},
	volume         = {159},
	pages          = {133--154},
	doi            = {10.61091/ars159-12}
}

@article{AvisHoang2023,
	title          = {On Reconfiguration Graphs of Independent Sets Under Token Sliding},
	author         = {David Avis and Duc A. Hoang},
	year           = {2023},
	journal        = {Graphs and Combinatorics},
	volume         = {39},
	number         = {3},
	doi            = {10.1007/s00373-023-02644-w},
	eid            = {59}
}

@article{Hoang2026,
	title          = {On Realizing Reconfiguration Graphs of Cliques},
	author         = {Hoang, Duc A.},
	year           = {2026},
	journal        = {arXiv preprint},
	archiveprefix  = {arXiv},
	eprint         = {2604.03567}
}

@inproceedings{LamPhanHoang2026,
	title          = {A Note on Reconfiguration Graphs of Cliques},
	author         = {Lam, Nhat-Quan and Phan, Huu-An and Hoang, Duc A.},
	year           = {2026},
	booktitle      = {Proceedings of CALDAM 2026},
	series         = {LNCS},
	volume         = {16445},
	pages          = {416--427},
	doi            = {10.1007/978-3-032-17156-6_31},
	editor         = {Misra, Neeldhara and Pandey, Arti}
}

@article{Nishimura2018,
	title          = {Introduction to Reconfiguration},
	author         = {Nishimura, Naomi},
	year           = {2018},
	journal        = {Algorithms},
	volume         = {11},
	number         = {4},
	doi            = {10.3390/a11040052},
	eid            = {52}
}

@incollection{vandenHeuvel2013,
	title          = {The Complexity of Change},
	author         = {Jan van den Heuvel},
	editor         = {Blackburn, Simon R. and Gerke, Stefanie and Wildon, Mark},
	year           = {2013},
	booktitle      = {Surveys in Combinatorics},
	publisher      = {Cambridge University Press},
	series         = {London Mathematical Society Lecture Note Series},
	volume         = {409},
	pages          = {127--160},
	doi            = {10.1017/cbo9781139506748.005}
}

@incollection{MynhardtN2019,
	title          = {Reconfiguration of colourings and dominating sets in graphs},
	author         = {Mynhardt, C.M. and Nasserasr, S.},
	year           = {2019},
	booktitle      = {50 years of Combinatorics, Graph Theory, and Computing},
	publisher      = {CRC Press},
	pages          = {171--191},
	doi            = {10.1201/9780429280092-10},
	editor         = {Fan Chung and Ron Graham and Frederick Hoffman and Ronald C. Mullin and Leslie Hogben and Douglas B. West},
	edition        = {1st}
}

@article{BousquetMNS2024,
	title          = {A survey on the parameterized complexity of reconfiguration problems},
	author         = {Nicolas Bousquet and Amer E. Mouawad and Naomi Nishimura and Sebastian Siebertz},
	year           = {2024},
	journal        = {Computer Science Review},
	volume         = {53},
	doi            = {10.1016/j.cosrev.2024.100663},
	note           = {(article 100663)}
}

@article{HearnDemaine2005,
	title          = {{PSPACE}-Completeness of Sliding-Block Puzzles and Other Problems through the Nondeterministic Constraint Logic Model of Computation},
	author         = {Robert A. Hearn and Erik D. Demaine},
	year           = {2005},
	journal        = {Theoretical Computer Science},
	volume         = {343},
	number         = {1-2},
	pages          = {72--96},
	doi            = {10.1016/j.tcs.2005.05.008}
}

@article{MessingerPipes2026,
	title          = {On Pancyclicity in a Mixed Model for Domination Reconfiguration},
	author         = {Margaret-Ellen Messinger and Logan Pipes},
	year           = {2026},
	journal        = {Discrete Mathematics \& Theoretical Computer Science},
	volume         = {28},
	number         = {2},
	doi            = {10.46298/dmtcs.15637}
}

@article{MessingerPorter2025,
	title          = {Eulerian $k$-dominating reconfiguration graphs},
	author         = {M. E. Messinger and A. Porter},
	year           = {2025},
	journal        = {Discrete Mathematics \& Theoretical Computer Science},
	volume         = {27},
	number         = {2},
	doi            = {10.46298/dmtcs.13438}
}

@article{ErohSchultz1998,
	title          = {Matching graphs},
	author         = {Eroh, Linda and Schultz, Michelle},
	year           = {1998},
	journal        = {Journal of Graph Theory},
	volume         = {29},
	number         = {2},
	pages          = {73--86},
	doi            = {10.1002/(SICI)1097-0118(199810)29:2<73::AID-JGT3>3.0.CO;2-9}
}

@article{JonesRoehmSchultz1998,
	title          = {On Matchings in Graphs},
	author         = {Jones, D. M. and Roehm, D. J. and Schultz, Michelle},
	year           = {1998},
	journal        = {Ars Combinatoria},
	volume         = {50},
	pages          = {65--79}
}

@article{WangYuanLiuLin1998,
	title          = {On the maximum matching graph of a graph},
	author         = {Wang, Shi-ying and Yuan, Jin-jiang and Liu, Yan and Lin, Yixun},
	year           = {1998},
	journal        = {Operations Research Transactions},
	volume         = {2},
	number         = {2},
	pages          = {13--17},
	doi            = {10.15960/j.cnki.issn.1007-6093.1998.02.002}
}

@article{Liu2004Subgraphs,
	title          = {Subgraphs of Maximum Matching Graphs},
	author         = {Liu, Yan},
	year           = {2004},
	journal        = {Indian Journal of Pure and Applied Mathematics},
	volume         = {35},
	number         = {9},
	pages          = {1063--1067}
}

@article{Liu2005MaxMatchingTypes,
	title          = {Characterizations of maximum matching graphs of certain types},
	author         = {Liu, Yan},
	year           = {2005},
	journal        = {Discrete Mathematics},
	volume         = {290},
	number         = {2-3},
	pages          = {283--289},
	doi            = {10.1016/j.disc.2004.09.011}
}

@article{LiuLiu2014SecondKind,
	title          = {Second kind maximum matching graph},
	author         = {Liu, Yan and Liu, Zhengbiao},
	year           = {2014},
	journal        = {Discrete Mathematics},
	volume         = {323},
	pages          = {27--34},
	doi            = {10.1016/j.disc.2014.01.011}
}

@article{FabilaMonroyFHHUW2012,
	title          = {Token Graphs},
	author         = {Ruy Fabila Monroy and David Flores{-}Pe{\~{n}}aloza and Clemens Huemer and Ferran Hurtado and Jorge Urrutia and David R. Wood},
	year           = {2012},
	journal        = {Graphs and Combinatorics},
	volume         = {28},
	number         = {3},
	pages          = {365--380},
	doi            = {10.1007/s00373-011-1055-9}
}

@article{AberleGMO2025,
	title          = {Cycle decompositions of Cartesian products of two cycles},
	author         = {Aberle, Moriah and Gold, Sarah and Moshe, Rivkah and Offner, David},
	year           = {2025},
	journal        = {Graphs and Combinatorics},
	volume         = {41},
	doi            = {10.1007/s00373-025-02953-2},
	eid            = {84}
}

@article{FernandesLPSZ2026,
	title          = {A Study on Token Digraphs},
	author         = {Fernandes, Cristina G. and Lintzmayer, Carla N. and Pe{\~n}a, Juan P. and Santos, Giovanne and Trujillo-Negrete, Ana and Zamora, Jose},
	year           = {2026},
	journal        = {Discrete Mathematics},
	volume         = {349},
	number         = {5},
	doi            = {10.1016/j.disc.2025.114951},
	eid            = {114951}
}

@book{Prisner1995GraphDynamics,
	title          = {Graph Dynamics},
	author         = {Prisner, Erich},
	year           = {1995},
	publisher      = {Longman},
	address        = {Harlow, Essex},
	series         = {Pitman Research Notes in Mathematics Series},
	volume         = {338},
	isbn           = {9780582286962}
}

\clearpage
\appendix

\section{Prior token-sliding realizability results}
\label{app:known-realizability}

In \Cref{tab:known-realizability}, we summarize some previous known results.
Rows cited in the main text are used as dependencies; the remaining rows place the present constructions beside known classifications.
In the table, $H=(K\cup S,E)_{\text{$K$-max}}$ denotes a connected split graph whose clique part $K$ is a maximum clique and whose stable part is~$S$.
The notation $D_{r,n,s}$ denotes the graph obtained from $P_n$ by attaching $r$ leaves to one end and $s$ leaves to the other.
The forest row involving $T\sqcup rK_1$ uses the source-specific choice of the number of isolated vertices needed for the realization.

\begin{table}[htbp]
\centering
\renewcommand{\arraystretch}{1.18}
\begin{tabularx}{\textwidth}{@{}L{0.20\textwidth}YYL{0.11\textwidth}@{}}
\toprule
Group & Target graph $H$ & $\TSsym_k$-realizable? & Reference \\
\TableGroupRule
\multirow[c]{5}{0.20\textwidth}{Basic graphs}
& $K_n$ ($n\ge 2$), $P_n$ ($n\ge 1$), and $C_n$ ($n\ge 3$) & yes for all $k\ge 2$ & \cite{AvisHoang2023} \\
\cmidrule(lr){2-4}
& $K_{m,n}$, $1\le m\le n$ & yes iff $(m=1\text{ and }n\le k)$ or $m=n=2$ & \cite{AvisHoang2023} \\
\cmidrule(lr){2-4}
& connected split graphs $H=(K\cup S,E)_{\text{$K$-max}}$ & yes iff $|N_H(v)\cap S|\le k-1$ for $v\in K$ and $\deg_H(w)=1$ for $w\in S$ & \cite{AvisHoang2023} \\
\cmidrule(lr){2-4}
& maximal outerplanar graphs on $n$ vertices and $K_n-e$ & yes iff $n\le 3$ & \cite{AvisHoang2023} \\
\cmidrule(lr){2-4}
& disjoint unions; graphs on at most three vertices & yes for all graphs on at most three vertices; closed under disjoint union & \cite{AvisHoang2023} \\
\TableGroupRule
\multirow[c]{6}{0.20\textwidth}{Forests/trees}
& $K_{1,n}$ & yes iff $n\le k$ & \cite{AvisHoang2023,AvisHoang2024} \\
\cmidrule(lr){2-4}
& $q$-ary trees, $q\ge 2$ & yes for $k=q+1$ & \cite{AvisHoang2024} \\
\cmidrule(lr){2-4}
& $T\sqcup rK_1$, where $T$ is any tree and $r\ge 0$ is chosen suitably & yes for $k=2$ & \cite{AvisHoang2024} \\
\cmidrule(lr){2-4}
& $D_{1,n,2}$ & yes iff $n=3$ for $k=2$ & \cite{AvisHoang2024} \\
\cmidrule(lr){2-4}
& $D_{r,2,s}$, $1\le r\le s$ & yes iff $s\le k-1$ & \cite{AvisHoang2024} \\
\cmidrule(lr){2-4}
& $D_{r,n,s}$, $n,k\ge 2$, $1\le r\le s\le k-1$ & yes; for $k=2$ this is the path case, and for $k\ge 3$ it is covered by the $(k-1)$-ary tree construction & \cite{AvisHoang2024} \\
\bottomrule
\end{tabularx}
\caption{Prior $\TSsym_k$-realizability results.}
\label{tab:known-realizability}
\end{table}

\Cref{tab:matching-graph-targets} lists some results for maximum matching graphs.
Recall that maximum matching graphs form a restricted class of independent-set token-sliding reconfiguration graphs: maximum matchings of a graph correspond to maximum independent sets in its line graph.
In the table, $Q_m$ denotes the $m$-dimensional hypercube.
A triangle-block graph is a graph whose every block is a triangle.
For graphs $X$ and $Y$, the join $X+Y$ is obtained from their disjoint union by adding all edges between $V(X)$ and $V(Y)$.
The $r$-th power of $C_n$ has vertex set $V(C_n)$, where two vertices are adjacent if their distance in $C_n$ is at most~$r$.
An antipodal path for a $C_4$ is an additional path joining two opposite vertices of the $C_4$.

\begin{table}[htbp]
\centering
\begin{tabularx}{\textwidth}{@{}L{0.42\textwidth}YL{0.14\textwidth}@{}}
\toprule
Target graph & Maximum matching graph? & Reference \\
\midrule
$K_m$, $P_m$, $K_{1,m}$, $Q_m$ & yes & \cite{JonesRoehmSchultz1998} \\
\cmidrule(lr){1-3}
$C_m$, $m\ge 3$ & yes & \cite{JonesRoehmSchultz1998,ErohSchultz1998} \\
\cmidrule(lr){1-3}
trees & yes & \cite{Liu2005MaxMatchingTypes} \\
\cmidrule(lr){1-3}
connected triangle-block graphs, all degrees in $\{2,4\}$ & yes & \cite{WangYuanLiuLin1998} \\
\cmidrule(lr){1-3}
complete multipartite graphs & yes iff $C_4$, a star, or complete & \cite{ErohSchultz1998} \\
\cmidrule(lr){1-3}
$C_n^r$, $r\ge 2$, $n\ge 2r+2$ & no & \cite{ErohSchultz1998} \\
\cmidrule(lr){1-3}
$X+Y$, $X,Y$ connected and not both complete & no & \cite{ErohSchultz1998} \\
\cmidrule(lr){1-3}
graphs containing induced $K_4-e$ or $K_{2,k}$, $k\ge 3$ & no & \cite{JonesRoehmSchultz1998,Liu2004Subgraphs} \\
\cmidrule(lr){1-3}
graphs containing $C_4$ with an antipodal path of length $3$ or $4$ & no & \cite{ErohSchultz1998} \\
\bottomrule
\end{tabularx}
\caption{Maximum-matching-graph results.}
\label{tab:matching-graph-targets}
\end{table}
\end{document}